%% file: boltzmann_article.tex
\documentclass[onefignum,onetabnum]{siamart190516}

\input{boltzmann_shared}

\ifpdf
\hypersetup{
  pdftitle={An Entropic Method for Discrete Systems with Gibbs Entropy},
  pdfauthor={Z. Cai, J. Hu, Y. Kuang, and B. Lin}
}
\fi

\makeatletter
\newcommand*{\addFileDependency}[1]{% argument=file name and extension
  \typeout{(#1)}
  \@addtofilelist{#1}
  \IfFileExists{#1}{}{\typeout{No file #1.}}
}
\makeatother

\newcommand\boldf{\boldsymbol{f}}

\newcommand\boldC{\boldsymbol{C}}
\newcommand\boldone{\boldsymbol{1}}
\newcommand\boldg{\boldsymbol{g}}
\DeclareMathOperator*{\argmin}{argmin}
\newtheorem{example}{Example}

\begin{document}

\maketitle

% REQUIRED
\begin{abstract}
  We consider general systems of ordinary differential equations with monotonic Gibbs entropy, and introduce an entropic scheme that simply imposes an entropy fix after every time step of any existing time integrator. It is proved that in the general case, our entropy fix has only infinitesimal influence on the numerical order of the original scheme, and in many circumstances, it can be shown that the scheme does not affect the numerical order. Numerical experiments on the linear Fokker-Planck equation and nonlinear Boltzmann equation are carried out to support our numerical analysis.
\end{abstract}

% REQUIRED

% REQUIRED

\section{Introduction}
The second law of thermodynamics, discovered more than 170 years ago, states that the direction of the thermodynamic processes is driven by a physical quantity called entropy. The importance of this law cannot be overstated, and nearly every thermodynamic model has to respect such a property. Mathematically, there are a number of formulas to represent the entropy, among which the Gibbs entropy, formulated as the integral of $f \log f$ with $f$ being the distribution function of the states, is widely used in a variety of models such as the heat equation, the Boltzmann equation, and the Fokker-Planck equation. In our discussion, we assume a finite number of states, so that the Gibbs entropy is defined by
\begin{displaymath}
	\eta(\boldf) = \sum_{i=1}^N f_i \log f_i \Delta v_i,
\end{displaymath}
where $\boldf = (f_1, \ldots, f_N)^T \in \mathbb{R}_+^N$ describes the distribution of the $N$ states and $\Delta v_i$ represents the weight of the $i$th state. The vector $\boldf$ is a vector function of time $t$, and we assume that it satisfies the initial value problem
\begin{equation} \label{eq:assumesystem}
    \begin{aligned}
    &\frac{\mathrm{d}f_i(t)}{\mathrm{d}t} = Q_i(\boldf(t)), &i = 1,\dots, N, \\
    &f_i(0) = f_i^0, &
    \end{aligned}
\end{equation}
%And we are interested in a conservative, positive and entropic system, which can be characterized by the following hypotheses:
with the following properties:
\begin{enumerate}[label=(P\arabic*)]
    \item \label{po1} conservation of mass: $\frac{\mathrm{d}}{\mathrm{d}t} \sum_{i=1}^N f_i(t) \Delta v_i = 0$;
    \item \label{po2} nonnegativity: $f_i(t) \geq 0, \ \forall 1 \leq i \leq N, t \geq 0$;
    \item \label{po3} monotonicity of entropy: $\frac{\mathrm{d}}{\mathrm{d}t} \sum_{i=1}^N f_i(t) \log f_i(t) \Delta v_i \leq 0$.
\end{enumerate}
The ODE system of the form (\ref{eq:assumesystem}) appears frequently after discretizing the thermodynamic equations in space. For example, it may arise from the finite difference discretization of the heat equation and the Fokker-Planck type equation \cite{degond1994entropy, Buet2006numerical,Pareschi2018structure, Chow2019entropy}. It may also result from the discrete velocity method and the entropic Fourier method for the Boltzmann equation \cite{Goldstein1989investigations, Cai2018entropic}. 

Although the semi-discrete scheme (\ref{eq:assumesystem}) decays entropy, there is no guarantee that this property will carry over when time is discretized. In some special cases, the entropy decay can be proved for the fully discrete scheme, see for instance \cite{BCH20}, yet it often comes at a price of using implicit schemes and is highly problem and scheme dependent. Given the importance of entropy in thermodynamic processes, it would be desirable to have a fully discrete entropic scheme that is generic (e.g., does not require a specific type of time discretization) as well as easily implementable (e.g., does not require expensive nonlinear iterations). 

%Numerically, it is also desired to preserve the entropic property of the models. For example, in the hyperbolic equations, entropic solutions guarantee the convergence to the physical weak solution \cite{LeFloch2002fully, Tadmor2016entropy}. As for the Gibbs entropy, due to the convexity of the entropy functional, the backward Euler scheme will automatically preserve the monotonicity of the entropy \cite{Liu2012entropy}. However, when the right-hand side of \eqref{eq:assumesystem} is nonlinear, explicit schemes are sometimes preferred to avoid solving nonlinear equations, and thus researchers have been trying to develop entropic schemes for explicit time integrators \cite{Buet1998conservative}. \zc{Can we find more references?} As far as we know, such schemes in the literature are only designed for specific equations. 

To bridge the above gap, we introduce an entropic scheme in this paper to achieve the following: one can apply any time discretization to the system \eqref{eq:assumesystem} as long as it maintains the mass conservation and nonnegativity of the solution. After each time step, if the entropy goes in the wrong direction, we provide a simple fix to make it decay monotonically. Such a fix is done by a weighted average of the current solution and the solution with maximum entropy. Via numerical analysis, we show that such a fix has only a tiny effect on the order of accuracy, and in various cases, it can be proven that the order of accuracy is not affected at all. Numerical experiments on the linear Fokker-Planck equation and nonlinear Boltzmann equation will also be carried out to support our findings.

% The outline is not required, but we show an example here.
The paper is organized as follows. In \cref{sec:main}, we first outline the procedure of our entropic method and summarize the main theorems of the method. The detailed proof of the theorems with some deeper understandings is illustrated in \cref{sec:entropy}. \cref{sec:experiments} provides the numerical experiments, and the conclusion follows in \cref{sec:conclusions}.

%Our main results are in \cref{sec:main}, our new algorithm is in \cref{sec:alg}, experimental results are in \cref{sec:experiments}, and the conclusions follow in \cref{sec:conclusions}.

\section{Main results}
\label{sec:main}

This section outlines the overall procedure of our entropic method and lists the main results of our numerical analysis. Before stating our theorems, we %would first like to 
introduce the notations and review some basic properties of the Gibbs entropy.

\subsection{Brief review of Gibbs entropy}

%\jh{Do we really need to introduce $H$ this early? Looks like you are using $\eta$ in the following theorems (at least in section 2.2). Also in the numerical sections, $h$ is used as spatial mesh size. It is better to distinguish them.}

Due to the conservation hypothesis \ref{po1}, below we focus on the entropy functional defined by
\begin{displaymath}
   H(\boldf) = \sum_{i=1}^N (f_i \log f_i - f_i)\Delta v_i:= \sum_{i=1}^N h(f_i) \Delta v_i ,
\end{displaymath}
with $h(x) = x \log x -x $. Note that $H(\boldf)$ differs from $\eta(\boldf)$ only by a constant.

Let $\boldC = (C, \dots, C)^T \in \mathbb{R}_+^N$ with
\begin{equation} \label{eq:mean}
	C = \frac{\sum_{i=1}^{N} f_i \Delta v_i}{\sum_{i=1}^{N} \Delta v_i}.
\end{equation}
We denote by $\tilde{\boldf} = {\boldf}/{C} =  (\tilde{f}_1, \dots, \tilde{f}_N)^T$ the normalized $\boldf$, then it can be checked that
\begin{equation} \label{eq:relative}
	C \eta (\tilde{\boldf}) = H(\boldf) - H(\boldC).
\end{equation}
Furthermore, we define the $L^p$ ($p=1,2$) norm and $L^{\infty}$ norm of any $\boldf$ as
\begin{displaymath}
	\| \boldf \|_p  = \left( \sum_{i=1}^N  f_i^p \Delta v_i \right)^{1/p}, \qquad \| \boldf \|_{\infty} = \max_{i} |f_i|.
\end{displaymath}

\begin{lemma} \label{lm:cminf}
	$\boldC$ is the unique global minimum point of $H(\boldf)$ for all $\boldf \in \mathbb{R}_+^N$ satisfying \cref{eq:mean} with fixed $C$. 
\end{lemma}

%\begin{proof}
%	The proof utilizes the concavity of $\log(x)$ and Jensen's inequality.
%	
%	Given random variable $x_i \sim \frac{f_i}{NC}$, $1 \leq i \leq N$, by Jensen's inequality,
%	\begin{displaymath}
%		\frac{1}{N} \sum_{i=1}^N \frac{f_i}{C} \left( \log x_i \right) \leq \log \left( \frac{1}{N} \sum_{i=1}^N \frac{f_i}{C}x_i \right) .
%	\end{displaymath}
%	Take $x_i = \frac{C}{f_i}$ in the equation above for $1 \leq i \leq N$, it becomes
%	\begin{displaymath}
%		\frac{1}{N} \sum_{i=1}^N \frac{f_i}{C} \left( \log C - \log f_i \right) \leq 0,
%	\end{displaymath}
%	which could be multiplied to $\Delta v NC$ and reformulated as
%	\begin{displaymath}
%		\Delta v \sum_{i=1}^N f_i \log f_i  \geq \Delta v N C \log C.
%	\end{displaymath}
%	Subtracting both sides by $\Delta v \sum_{i=1}^N f_i$, we get the final result
%	\begin{displaymath}
%		H(\boldf)  \geq \Delta v N C (\log C - 1) = H(\boldC).
%	\end{displaymath}
%	In the above inequality, the equality holds if and only if $f_1 = f_2 = \cdots = f_N$, which comes from the condition of Jensen's inequality (i.e., $x_1 = x_2 = \cdots = x_N$). For fixed $C$, the equality holds if and only if $\boldf = \boldC$.
%\end{proof}

The proof of \cref{lm:cminf} can be done by the concavity of $\log(x)$ and Jensen's inequality. Furthermore, a straightforward corollary of \cref{lm:cminf} is that, $\boldsymbol{1}=(1, \dots, 1)^T \in \mathbb{R}_+^N$ is the unique global minimum point of $\eta(\tilde{\boldf})$ for all $\tilde{\boldf} \in \mathbb{R}_+^N$ satisfying $\| \tilde{\boldf} \|_1= \| \boldsymbol{1} \|_1$. To ease the notation, we use $ \| \boldone \|_{1} = \sum_{i=1}^N \Delta v_i = V$ to denote the volume. 
%\begin{lemma} \label{lm:hpositive}
%	$\boldsymbol{1}=(1, \cdots, 1)^T \in \mathbb{R}_+^N$ is the only global minimum point of $H[\tilde{\boldf}]$ for fixed C. Therefore, $H[\tilde{\boldf}] \geq H[\boldone] = 0$.
%\end{lemma}

The notations hereafter will be focused on the relative entropy $\eta(\tilde{\boldf})$ and the normalized $\tilde{\boldf}$ for fixed $C$. One could find its relationship to entropy function $H(\cdot)$ from \cref{eq:relative}. For simplicity, we would like to omit the tilde symbol in $\tilde{\boldf}$, and thus the average of the components of $\boldf$ will be $1$ hereafter. 

\subsection{Main results}
We assume after temporal discretization of \cref{eq:assumesystem}, the properties \ref{po1} and \ref{po2} can be preserved. Specifically, if we let $\boldf^n\geq 0$ be the numerical solution at the $n$th time step, then we have
\begin{enumerate}[label=(H\arabic*)]
    \item \label{hypo1} conservation: $\sum_{i=1}^N f_i^{n+1} \Delta v_i = \sum_{i=1}^N f_i^n \Delta v_i$,
    \item \label{hypo2} nonnegativity: $f_i^{n+1} \geq 0, \ \forall 1 \leq i \leq N$.
\end{enumerate}
%A number of numerical schemes have been proposed to satisfy these two properties (e.g. \cite{Zhang2010maximum, Pareschi2000stability}), and here 
We would like to
design an entropic method such that it can fulfill a discrete version of \ref{po3} while keeping \ref{hypo1} and \ref{hypo2}. %\ykrevise{Based on these schemes, we would like to deliver an entropic method such that the numerical schemes can fulfill a discrete version of \ref{po3} while keeping \ref{hypo1} and \ref{hypo2}.}

Our numerical scheme is based on imposing a simple entropy fix after computing the numerical solution at every time step. Suppose that $\boldf^{n+1}$ is computed through evolving $\boldf^n$ by one time step. If $\eta(\boldf^{n+1}) \leq \eta(\boldf^n)$, nothing needs to be done.  Otherwise, we revise the solution at the $(n+1)$th time step as
\begin{equation} \label{eq:fixedsol}
	\hat{\boldf}^{n+1} = \boldf^{n+1} + \beta_p ( \boldone - \boldf^{n+1}),
\end{equation}
where $\beta_p \in (0,1]$ is chosen to satisfy
\begin{equation} \label{eq:enfixpratic}
	\eta(\boldf^{n+1} + \beta_p( \boldone - \boldf^{n+1})) = \eta (\boldf^n).
\end{equation}
This guarantees that the entropy is always non-increasing.

In most cases, such a method stabilizes the solution since it reduces both the Gibbs entropy and the $2$-norm of vectors. Therefore we are mainly concerned about the magnitude of the fixing term $\beta_p (\boldone - \boldf^{n+1})$, and we hope that this term does not affect the numerical convergence order of the original scheme. Generally, the error estimation of this scheme can be analyzed in the following manner
\begin{equation} \label{eq:error}
    \begin{aligned}
        \| \hat{\boldf}^{n+1} - \boldf(t_{n+1})\| & \leq \| \hat{\boldf}^{n+1} - \boldf^{n+1} \| + \| \boldf^{n+1} - \boldf(t_{n+1}) \| \\
        & \leq \| \hat{\boldf}^{n+1} - \boldf^{n+1} \| + \| \boldf^{n+1} - \tilde{\boldf}(t_{n+1}) \| + \| \tilde{\boldf}(t_{n+1}) - \boldf(t_{n+1}) \|,
    \end{aligned}
\end{equation}
where $\tilde{\boldf}(t)$ is the solution of the problem
\begin{equation} \label{eq:local_ode}
    \begin{aligned}
    &\frac{\mathrm{d}\tilde{f}_i(t)}{\mathrm{d}t} = Q_i(\tilde{\boldf}(t)), &i = 1,\dots, N, \\
    &\tilde{f}_i(t_n) = f_i^n, &i = 1,\dots,N,
    \end{aligned}
\end{equation}
and hence $\|\boldf^{n+1} - \tilde{\boldf}(t_{n+1})\|$ is the ``one-step error'' of the scheme. The last term in \cref{eq:error} is usually controlled by the stability of the ODE problem with respect to the initial condition.
If we assume that the scheme satisfies the following consistency condition:
\begin{displaymath}
\| \boldf(t_{n+1}) - \boldf^{n+1} \| \leq O(\Delta t ^{s+1}),
\end{displaymath}
then the original scheme (before our entropy fix) is a scheme of order $s$. Here our purpose is to demonstrate that the first term in the second line of \cref{eq:error}, i.e., $\|\beta_p (\boldone - \boldf^{n+1})\|$, can be controlled by the second term $\|\tilde{\boldf}(t_{n+1}) - \boldf^{n+1}\|$. In the ideal case, we may find a constant $C$ such that 
\begin{displaymath}
    \|\beta_p (\boldone - \boldf^{n+1})]\| \leq C \|\tilde{\boldf}(t_{n+1}) - \boldf^{n+1}\|,
\end{displaymath}
then the numerical convergence order is not affected. Hereafter, for simplicity, we would like to omit the tilde and use $\boldf(t_{n+1})$ to denote the solution of \cref{eq:local_ode} at time $t_{n+1}$. In other words, we assume that the solution at the $n$th time step $\boldf^n$ is exact ($\boldf(t_n) = \boldf^n$), so that $\boldf(t_{n+1})$ becomes identical to $\tilde{\boldf}(t_{n+1})$.

%where $\boldf(t_{n+1})$ is the exact solution of \cref{eq:assumesystem} at $(n+1)$-th step with initial value $\boldf^{n}$, and $\boldf^{n+1}$ is the numerical solution. Although the local consistency looks slightly different from the classical definition of local truncation error (of order $s$), they are actually identical if we suppose the exact solution at $n$-th time step is $\boldf^n$. %We would like to give a short remark here that for fixed $C$, the consistency of original $\boldf$ is identical to the consistency of normalized one.

%From the property of Boltzmann equation \bo{also here}, we know that $H[\boldf(t_{n+1})] \leq H[\boldf^n]$. If $H[\boldf^{n+1}] > H[\boldf^n]$, we try to revise the entropy by finding a (small) real number $0<\beta_p \leq 1$ such that
%We wish to show that in \cref{eq:enfixpratic}, the fix of the numerical solution $\| \beta_p  ( \boldone - \boldf^{n+1})\|_2$ can be controlled by the local truncation error, i.e., the difference between  $\boldf(t_{n+1})$ and $\boldf^{n+1}$,
%meaning that the entropic revision does not affect the consistency of numerical scheme. If so, we can set the numerical solution at $(n+1)$th step as
%and the numerical order can be retained.

In the following theorems, we will study a stronger result
\begin{equation} \label{eq:enfix}
	\eta(\boldf^{n+1} + \beta( \boldone - \boldf^{n+1})) = \eta(\boldf(t_{n+1})),
\end{equation}
where $\beta_p$ in \cref{eq:enfixpratic} is replaced by $\beta$ and the solution at $(n+1)$th time step is revised to possess the same entropy as $\boldf(t_{n+1})$.
Due to $\eta(\boldf(t_{n+1})) \leq \eta(\boldf^n)$ and the monotonicity of $\eta(\boldf^{n+1} + \omega ( \boldone - \boldf^{n+1}))$ with respect to $\omega$,  we see that $\beta_p \leq \beta$.
Therefore, it suffices to show that $\| \beta  ( \boldone - \boldf^{n+1})\|$ can be controlled by the difference between $\boldf(t_{n+1})$ and $\boldf^{n+1}$.
Based on the commonly-used $2$-norm of vectors, we are going to prove this type of results in four different scenarios, which will be stated in the four theorems listed below.

In the first case, we have no assumptions on the structure of the solution, which may lead to a slight reduction of the numerical convergence order:

\begin{theorem} \label{thmlog}
	Given a positive and conservative numerical scheme, i.e., $\boldf^{n+1} \in \mathbb{R}_+^N$ and $\| \boldf^{n+1}\|_1 = \| \boldf(t_{n+1})\|_1$. When $\eta(\boldf^{n+1}) > \eta(\boldf^n)$ and \cref{eq:enfix} are satisfied, if $\| \boldf(t_{n+1}) - \boldf^{n+1} \|_2 \leq 1$, then
	
	\begin{displaymath}
		\| \beta  ( \boldone - \boldf^{n+1})\|_2 \leq M \| \boldf(t_{n+1}) - \boldf^{n+1} \|_2 \left( 1 + \left| \log \left( \| \boldf(t_{n+1}) - \boldf^{n+1} \|_2 \right) \right| \right),
	\end{displaymath}
	where $M>0$ is a constant which depends on $V$, $\| \boldf^{n+1} \|_{\infty}$ and $\| \boldf(t_{n+1})\|_{\infty}$.
\end{theorem}

In this case, the right-hand side of the inequality contains a logarithmic term, which tends to infinity when $\|\boldf(t_{n+1}) - \boldf^{n+1}\|_2$ approaches zero. However, for any $\epsilon > 0$, we have
\begin{displaymath}
    1 + \left| \log \left( \| \boldf(t_{n+1}) - \boldf^{n+1} \|_2 \right) \right| < \|\boldf(t_{n+1}) - \boldf^{n+1}\|_2^{-\epsilon}   
\end{displaymath}
when $\|\boldf(t_{n+1}) - \boldf^{n+1}\|_2$ is sufficiently small, meaning that the numerical convergence order is reduced only by an arbitrary small positive number. Nevertheless, we would still like to explore the conditions under which such a logarithmic term does not exist. The remaining three cases are related to this type of results.

Intuitively, the reason of the logarithmic term in \cref{thmlog} is the unboundedness of the function $h'(x)$ when $x$ is close to zero. In the following result, we assume that the components of the numerical solution $\boldf^{n+1}$ have a lower bound $C_0$, such that $h'(x)$ becomes bounded:
%In the second case, we require $f_i^{n+1} \geq C_0>0$ for all $i$, where $\boldf^{n+1} = (f_1^{n+1}, \cdots, f_N^{n+1})^T$. This requirement makes the derivative of $H[\cdot]$ uniformly bounded. The result is given in the following theorem.

\begin{theorem} \label{thm1}
	Given a positive and conservative numerical scheme, i.e., $\boldf^{n+1} \in \mathbb{R}_+^N$ and $\| \boldf^{n+1}\|_1 = \| \boldf(t_{n+1})\|_1$. When $\eta(\boldf^{n+1}) > \eta(\boldf^n)$ and \cref{eq:enfix} are satisfied, if $f^{n+1}_i \geq C_0>0$ holds for all $1 \leq i \leq N$, then
	
	\begin{displaymath}
		\| \beta  ( \boldone - \boldf^{n+1})\|_2 \leq M \| \boldf(t_{n+1}) - \boldf^{n+1} \|_2,
	\end{displaymath}
	where $M>0$ is a constant which depends on $C_0$, $\| \boldf^{n+1} \|_{\infty}$ and $\| \boldf(t_{n+1})\|_{\infty}$.
\end{theorem}

%the condition ``$f_i^{n+1} \geq C_0>0$ for all $i$'' is relaxed, and the conclusion becomes weaker with additional logarithmic function. 

The condition in this theorem disallows the numerical solution to be zero anywhere in the domain. In such a situation, if the scheme can guarantee the numerical convergence order for the $L^{\infty}$-error, we can still show that the $L^2$-norm of the entropy fix is small. This corresponds to our third case:
\begin{theorem} \label{thm2}
	Given a positive and conservative numerical scheme, i.e., $\boldf^{n+1} \in \mathbb{R}_+^N$ and $\| \boldf^{n+1}\|_1 = \| \boldf(t_{n+1})\|_1$. When $\eta(\boldf^{n+1}) > \eta(\boldf^n)$ and \cref{eq:enfix} are satisfied, if $\| \boldf(t_{n+1}) - \boldf^{n+1} \|_{\infty} \leq 1/3$, it holds that
	
	\begin{displaymath}
		\| \beta  ( \boldone - \boldf^{n+1})\|_2 \leq M \| \boldf(t_{n+1}) - \boldf^{n+1} \|_{\infty},
	\end{displaymath}
	where $M>0$ is a constant which depends on $V$, $\| \boldf^{n+1} \|_{\infty}$ and $\| \boldf(t_{n+1})\|_{\infty}$.
\end{theorem}

The last case we consider can be regarded as a generalization of \cref{thm1}. We allow the numerical solution to be small on some part of the domain, but require that the solution increases slowly. This will lead to a result similar to the conclusion of \cref{thm1}, where the $L^2$-magnitude of the entropy fix can be directly bounded by the $L^2$-error:

\begin{theorem} \label{thm3}
	Given a positive and conservative numerical scheme, i.e., $\boldf^{n+1} \in \mathbb{R}_+^N$ and $\| \boldf^{n+1}\|_1 = \| \boldf(t_{n+1})\|_1$, we denote the components of $\boldf^{n+1}$ as $f_1^{n+1} \leq f_2^{n+1} \leq \cdots \leq f_N^{n+1}$. For any $C_1, C_f \in (0, 1]$, there exists two positive constants $\delta$ and $M$, such that
	\begin{displaymath}
		\| \beta  ( \boldone - \boldf^{n+1})\|_2 \leq M \| \boldf^{n+1} - \boldf(t_{n+1}) \|_{2}
	\end{displaymath}
	if all the following conditions hold:
	\begin{itemize}
	\item $\eta(\boldf^{n+1}) > \eta(\boldf^n)$ and $\eta(\boldf^{n+1} + \beta(\boldone - \boldf^{n+1})) = \eta(\boldf(t_{n+1}))$;
	\item $\| \boldf^{n+1} - \boldf(t_{n+1}) \|_2 < \delta$;
	\item The index $I_1 = \min \{ I \mid \sum_{i=1}^I \Delta v_i \geq C_1 V \}$ satisfies
	\begin{equation} \label{eq:maxlogratio}
		\frac{1}{| \log \left( f_1^{n+1} \right) |} \geq \frac{C_f}{|\log \left( f_{I_1}^{n+1} \right)|}.
	\end{equation}
	\end{itemize}
	Here $\delta$ depends on $C_1$, $C_f$ and $V$, and $M$ depends on
	$C_1$, $C_f$, $V$, $\| \boldf^{n+1} \|_{\infty}$ and $\| \boldf(t_{n+1})\|_{\infty}$. %When $H[\boldf^{n+1}] > H[\boldf^n]$ and \cref{eq:enfix} are satisfied, let $I_1 = \lceil C_1 N \rceil$, if $\| \boldf^{n+1} - \boldf(t_{n+1}) \|_2 < \delta$ and furthermore,
	%\begin{equation} \label{eq:maxlogratio}
	%	\frac{1}{| \log \left( f_1^{n+1} \right) |} \geq \frac{C_f}{|\log \left( f_{I_1}^{n+1} \right)|},
	%\end{equation}
	%% \frac{\log \left( f_{I_1}^{n+1} + \varepsilon (1 - f_{I_1}^{n+1}) \right)}{ \log \left( f_1^{n+1} + \varepsilon(1 - f_1^{n+1}) \right) } \right) \min \left( 1, C(C_1, C_f, V) \right)	for $I_1 = \lceil C_1 N \rceil$, where $0 < C_1, C_f \leq 1$ are constants independent of $N$, $\| \boldf^{n+1} - \boldf(t_{n+1}) \|_2 \leq \delta$ and $\delta = \delta(C_1, C_f, V)$ is a positive constant depending on $C_1$, $C_f$ and $V$, 
	%then	
	%\begin{displaymath}
	%	\| \beta  ( \boldone - \boldf^{n+1})\|_2 \leq M \| \boldf^{n+1} - \boldf(t_{n+1}) \|_{2}.
	%\end{displaymath}
	%where $M>0$ is a constant which depends on $C_1$, $C_f$, $V$, $\| \boldf^{n+1} \|_{\infty}$ and $\| \boldf(t_{n+1})\|_{\infty}$.
\end{theorem}

In \cref{eq:maxlogratio}, the function $1/|\log x|$ is regarded as zero when $x$ takes the value zero. The condition \cref{eq:maxlogratio}  allows the existence of small components in the solution. To better demonstrate the nature of this condition, two examples are presented below.

%there are two quotients in the $\max$ parenthesis, where the first one is critical and the second one helps to avoid the case when $f_1^{n+1} = 0$. If $f_1^{n+1} \neq 0$, we may just focus on the first quotient. On the other hand, if $f_1^{n+1} = 0$, we could do a $\varepsilon$ regularization on $\boldf^{n+1}$ as $\boldf^{n+1} + \varepsilon(\boldone - \boldf^{n+1})$. After the regularization, $f_1^{n+1} + \varepsilon(1 - f_i^{n+1}) \geq \varepsilon$ and $\| \varepsilon(\boldone - \boldf^{n+1}) \|_2 = O(\varepsilon)$, and then the first quotient in \cref{eq:maxlogratio} of regularized $\boldf^{n+1}$ is the second quotient of original $\boldf^{n+1}$.
\begin{example}\label{thmex:1}
This example assumes that $\boldf^{n+1}$ is the uniform discretization of a one-dimensional Gaussian, i.e.,
\begin{displaymath}
\Delta v_i = \Delta v, \ f_i^{n+1} = \frac{1}{C\sqrt{\pi}}\exp(- v_i^2), \qquad i = 1,\dots,N+1,
\end{displaymath}
where $v_i$ are uniformly distributed in $[-L, L]$, $\Delta v = 2L / (N+1)$ and $L > 0$ is set to be sufficiently large such that $\exp(-L^2)$ is sufficiently small. The constant $C$ is chosen such that $\|\boldf^{n+1}\|_1 = \|\boldone\|_1 $. Furthermore, fox fixed $L$, we assume that $N$ is an even number and large enough such that $C \geq {1}/({4L})$. According to the assumption of \cref{thm3}, we set $v_i$ to be
\begin{displaymath}
v_i = (-1)^i \lceil (N+1-i)/2 \rceil \frac{2 L}{N}, \qquad i = 1,\cdots,N+1
\end{displaymath}
such that $f_i^{n+1}$ increases with respect to $i$.
%Suppose $L >0$ is a constant indicating the radius of interval $[-L, L]$ such that $\frac{1}{\sqrt{\pi}} \exp(-L^2)$ is sufficiently small (e.g., $\frac{1}{\sqrt{\pi}} \exp(-L^2) \leq 10^{-15}$). We assume $[-L, L]$ is partitioned into $N$ subintervals with $N>2$ an even number, and $\boldf^{n+1}$ is the discretized Gaussian at $N+1$ endpoints $-L + 2 i L / N$ for $0 \leq i \leq N$. After normalized and sorted in the ascending order, the components of $\boldf^{n+1}$ are assumed to be $f_i^{n+1} = \frac{1}{C\sqrt{\pi}}\exp(- (v_i-L)^2)$ for $1 \leq i \leq N+1$, where $v_i = \frac{2 \lfloor (i-1)/2 \rfloor L}{N}$ and $C = \frac{1}{(N+1)\sqrt{\pi}}\sum_{i=1}^{N+1} \exp(- (v_i-L)^2)$.
For illustration, we plot the normalized Gaussian and its sorted version in \cref{fig:ex1}, where parameters are set as $L = 6$ and $N = 20$. In this example, we take $I_1 = \lceil (N + 1)/ 2 \rceil = N/2 + 1$, then
\begin{displaymath}
    \frac{\log (f_{I_1}^{n+1})}{\log (f_1^{n+1})} = \frac{\log\frac{1}{C\sqrt{\pi}} - v_{I_1}^2}{\log\frac{1}{C\sqrt{\pi}} - L^2} \geq \frac{v_{I_1}^2 - \log(\frac{4L}{\sqrt{\pi}})}{L^2} \geq \frac{v_{I_1}^2}{2 L^2} \geq \frac{1}{8},
\end{displaymath}
which satisfies \cref{eq:maxlogratio} with $C_1 = 1/2$ and $C_f = 1/8$. This example shows a case where the values of $f_i^{n+1}$ are nonzero but can be arbitrarily small.

\begin{figure}[htbp]
  \centering
  \includegraphics[width=0.48 \textwidth]{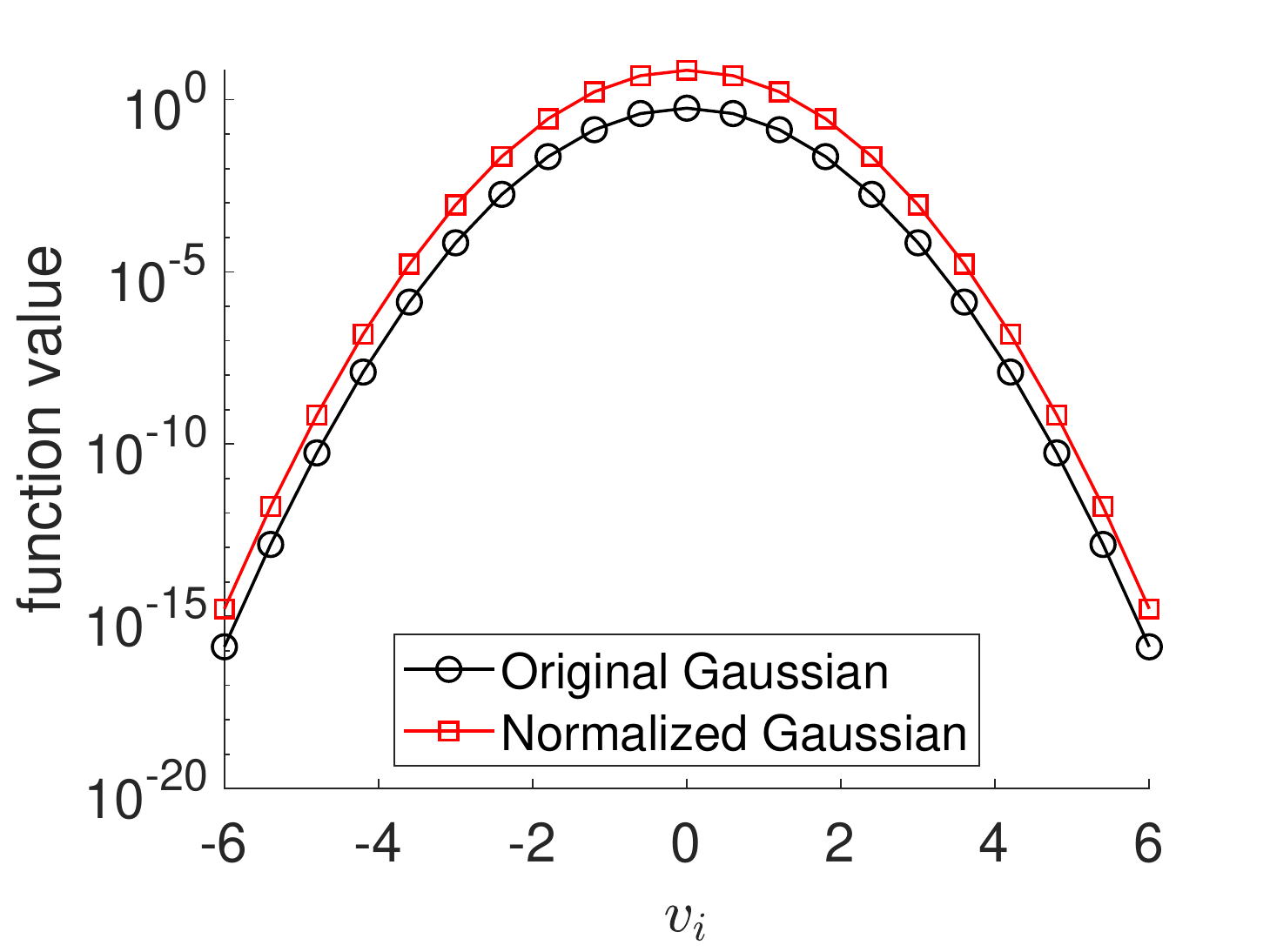}
  \includegraphics[width=0.48 \textwidth]{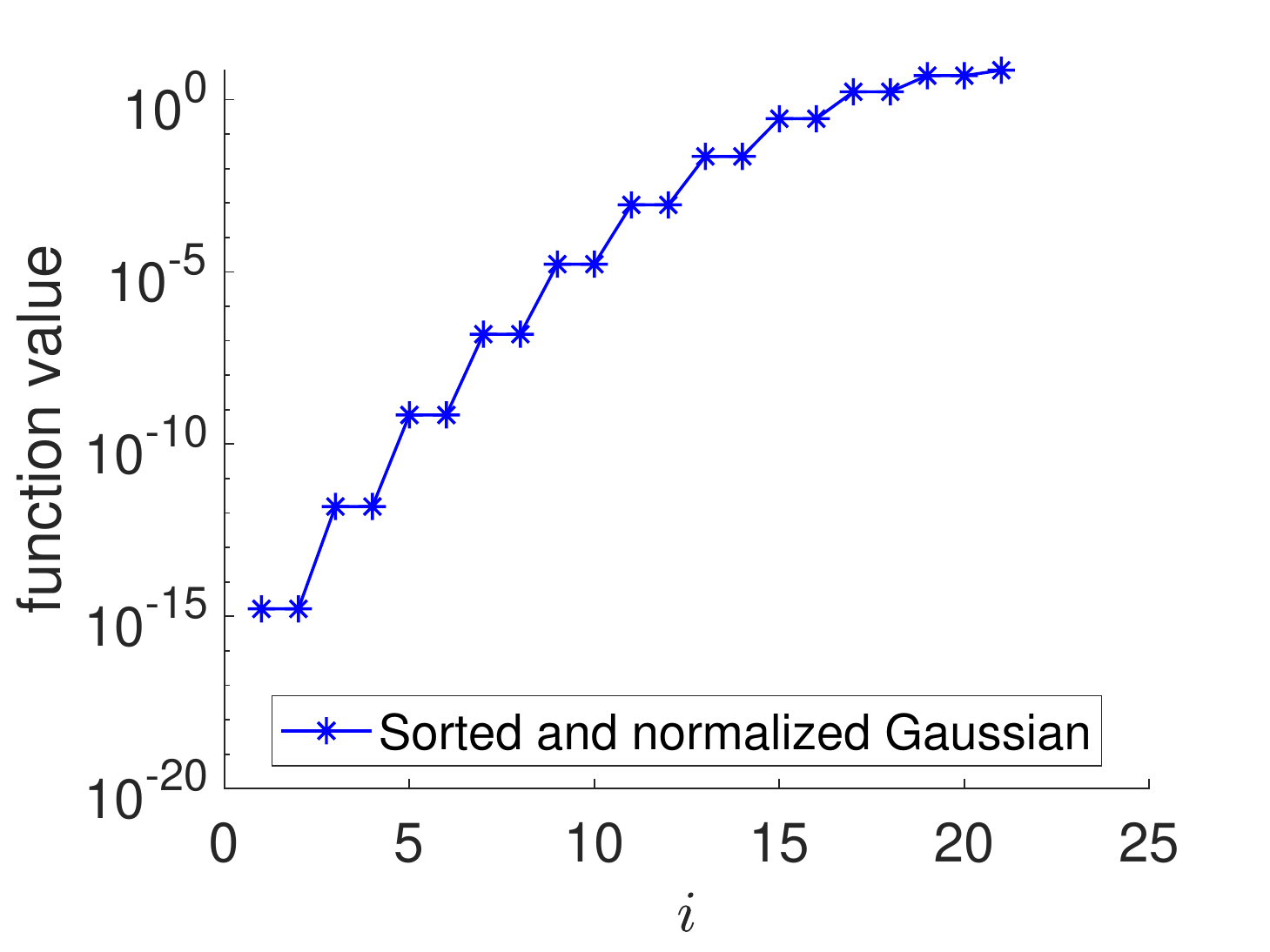}
  \caption{Discretized Gaussian, its normalization and sorted notation in \cref{thmex:1}.}
  \label{fig:ex1}
\end{figure}
\end{example}

\begin{example}\label{thmex:2}
The second example is for the case where some components of $\boldf^{n+1}$ are zero. We assume a uniform discretization on $[0,1]$ with $\Delta v_i = 1 / N$ for $i = 1, \dots, N$ and choose $\boldf^{n+1}$ to be
\begin{displaymath}
    f_i^{n+1} = \left\{ \begin{array}{ll}
        0, & i = 1,\dots,I_1, \\
        1, & i = I_1+1, \dots, N-I_1, \\
        2, & i = N-I_1+1, \dots, N.
    \end{array} \right.
\end{displaymath}
If $I_1/N$ is a constant, the vector $\boldf^{n+1}$ approximates a piecewise constant function. In this case, \cref{thm3} holds by choosing $C_1 = I_1/N$ and $C_f$ to be any positive number in $(0,1]$. The blue lines in \cref{fig:ex2} show the situation where $C_1 = 1/3$. However, if $I_1/N$ decreases to zero as $N$ increases, e.g. $I_1 \equiv 1$ for all $N$, such a constant $C_1$ cannot be found. This situation violates the condition of \cref{thm3}, which is illustrated as the red lines in \cref{fig:ex2}.

%The second example helps to illustrate $C_1$ is independent of $N$ and $f_{I_1}^{n+1} = 0$. In this example, $f_i^{n+1}$ can only be chosen as $0$, $1$ or $2$, where the number of $0$ is identical to the number of $2$ to make sure $\| \boldf^{n+1} \|_1 = \| \boldone \|_1$. For $N \gg 1$, if $f_1 = 0$, $f_N=2$ and $f_i = 1$ for $2 \leq i \leq N-1$, this is not applicable to \cref{eq:maxlogratio}. This is because the left-hand side $\frac{1}{| \log \left( f_1^{n+1} \right) |} =0$, therefore, right-hand side should be $0$, which means $f_{I_1}^{n+1} = 0$. Thus $I_1 = 1$, which cannot be represented as $\lceil C_1 N \rceil$ with $C_1$ independent of $N$. However, if $I_1 = \lceil N / 3 \rceil$, $f_1 = f_2 = \cdots = f_{I_1} =0$, $f_{I_1+1} = f_{I_1+2} = \cdots = f_{N-I_1} =1$ and $f_{N-I_1+1} = f_{N-I_1+2} = \cdots = f_{N} =2$, then \cref{eq:maxlogratio} is satisfied with $C_1 = 1 / 3$ and $C_f = 1$. This example is plotted in \cref{fig:ex2} for illustration.

\begin{figure}[htbp]
  \centering
  \includegraphics[width=0.48 \textwidth]{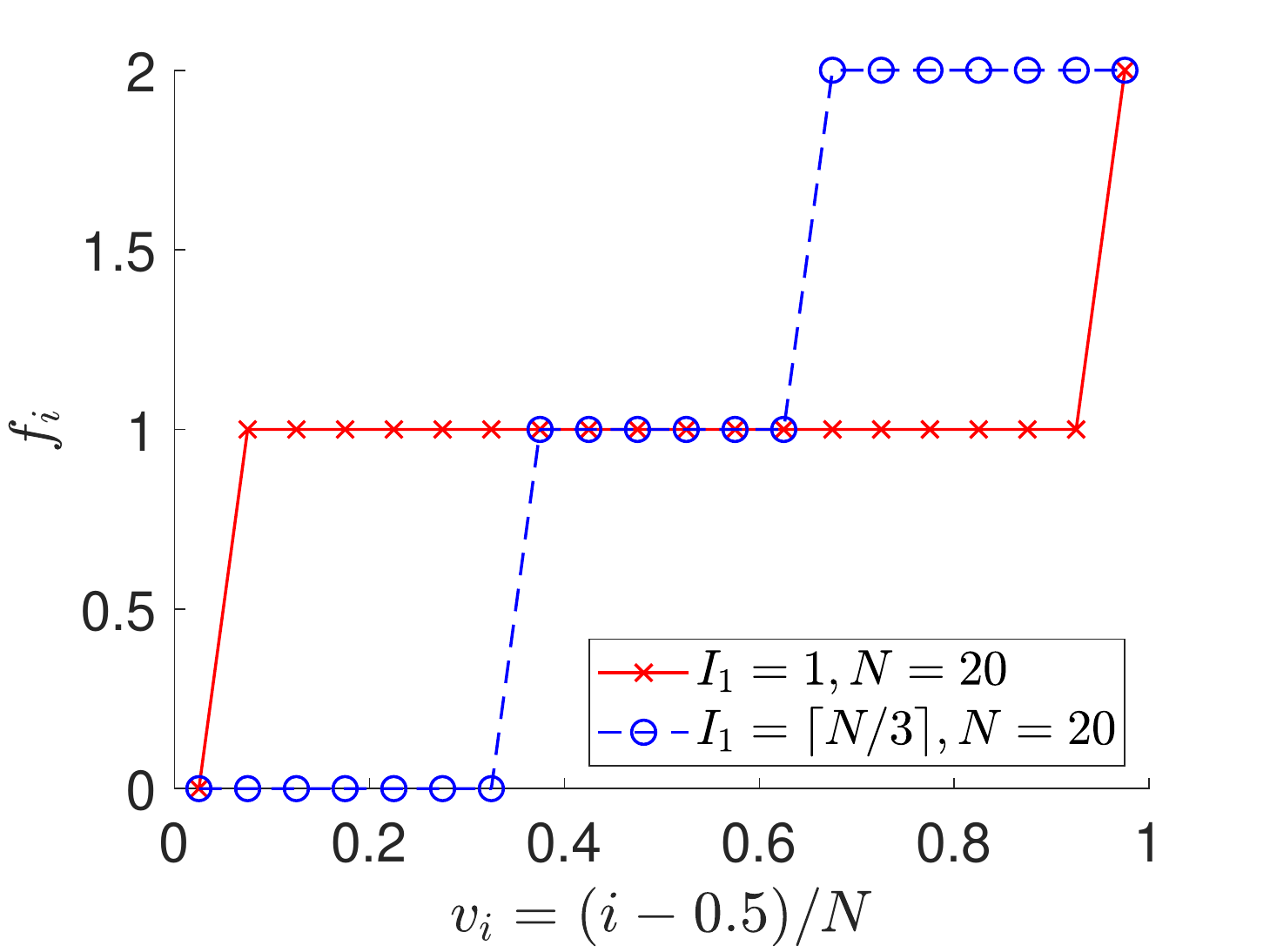}
  \includegraphics[width=0.48 \textwidth]{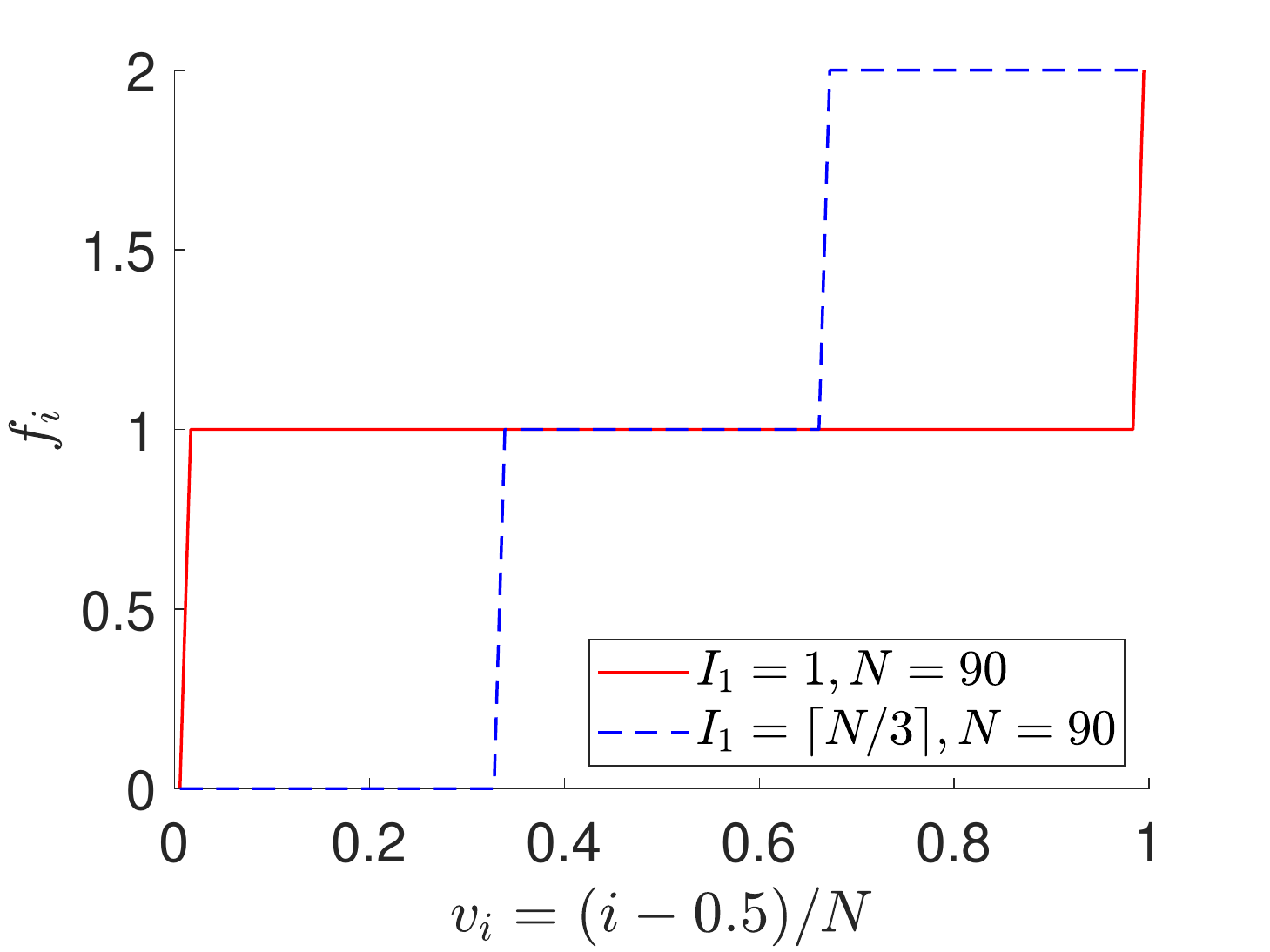}
  \caption{Illustration of \cref{thmex:2} where $f_i^{n+1}$ can only be chosen as $0$, $1$ or $2$.}
  \label{fig:ex2}
\end{figure}
\end{example}

In general, the above theorems suggest that such entropy fix can be safely used without sacrificing the numerical accuracy. Moreover, for a numerical scheme with sufficient accuracy, the violation of the entropy inequality will not always happen, meaning that the entropy fix may be needed only at a few time steps, resulting in even less significant impact on the numerical accuracy.

\begin{remark*}
The above results can be easily generalized to the cases where the equilibrium is not a constant. Assume that $\boldsymbol{\mathcal{M}} = (\mathcal{M}_1, \cdots, \mathcal{M}_N)^T \in \mathbb{R}_N^+$ is the equilibrium state of \cref{eq:assumesystem}, and the entropy functional (in this case, it is the relative entropy) is defined by
\begin{displaymath}
   \eta[\boldf] = \sum_{i=1}^N f_i \log \frac{f_i}{\mathcal{M}_i} \Delta v_i.
\end{displaymath}
We can let $g_i = f_i / \mathcal{M}_i$ and $\Delta w_i = \mathcal{M}_i \Delta v_i$, so that $\eta[\boldf]$ can be rewritten as
\begin{displaymath}
   \eta[\boldf] = \sum_{i=1}^N g_i \log g_i \Delta w_i,
\end{displaymath}
which fits the entropy formulas in the theorems again. In this case, the entropy fix \cref{eq:fixedsol} applied to $\boldg^{n+1}$ is equivalent to the following fix applied to $\boldf^{n+1}$:
\begin{equation}
	\hat{\boldf}^{n+1} = \boldf^{n+1} + \beta_p ( \boldsymbol{\mathcal{M}} - \boldf^{n+1}).
\end{equation}
By this transformation, our approach can also be applied to the linear Fokker-Planck equation. Please see the numerical section for more details.
\end{remark*}

\section{Theoretical proofs of the error estimates}
\label{sec:entropy}
This section provides all the details of the proofs of the four theorems. Instead of proving these theorems in the order they are presented, below we will first provide the proof of \cref{thm1}, which can provide necessary tools needed in the proof of \cref{thmlog}.

\subsection{Proof of \cref{thm1}}
Before proving the theorem, the relationship between entropy function and $L^2$ norm will be demonstrated by several lemmas. Among them, we will first estimate the entropy function $\eta(\boldf)$ and its $L^2$ norm $\| \boldf \|_2$ in the following lemma.

\begin{lemma} \label{lm:h=twonorm}
	For $\boldf \in \mathbb{R}_+^N$ and $\| \boldf \|_1 = V$,
	\begin{displaymath}
		\frac{1}{2 \| \boldf \|_{\infty} } \| \boldf - \boldone \|_2^2 \leq \eta(\boldf) \leq \| \boldf - \boldone \|_2^2.
	\end{displaymath}
	
	\begin{proof}
%    	Given $x >0$, by Taylor's theorem,
%    	\begin{displaymath}
%    		x \log x = (x - 1) + \int_1^x \frac{1}{t}(x - t) \mathrm{d}t.
%    	\end{displaymath}
%    	On one hand, \bo{Marked for myself: Mistake. Fix it tonight.}
%    	\begin{displaymath}
%    		\int_1^x \frac{1}{t}(x - t) \mathrm{d}t \leq \int_1^x \frac{1}{t} (x-1)\mathrm{d}t = (x-1)\log x \leq (x-1)^2,
%    	\end{displaymath}
%    	where the last inequality above utilizes $\log x \leq x-1$. 
        On one hand, for $x \geq 0$,
        \begin{displaymath}
           x \log x - (x - 1) \leq x (x-1) - (x-1) = (x-1)^2,
        \end{displaymath}
        where the inequality above uses $\log x \leq x-1$. On the other hand, by Taylor's theorem, 
        \begin{displaymath}
    		x \log x = (x - 1) + \int_1^x \frac{1}{t}(x - t) \mathrm{d}t.
    	\end{displaymath}
    	For $0 \leq x \leq \| \boldf \|_{\infty}$, the integral satisfies 
		\begin{displaymath}
			\int_1^x \frac{1}{t}(x - t) \mathrm{d}t \geq \int_1^x \frac{1}{\max (x, 1)}(x - t) \mathrm{d}t \geq \int_1^x \frac{1}{\| \boldf \|_{\infty}}(x - t) \mathrm{d}t = \frac{(x-1)^2}{2 \| \boldf \|_{\infty}}.
		\end{displaymath}
		Therefore,
		\begin{displaymath}
			(x-1) + \frac{(x-1)^2}{2 \| \boldf \|_{\infty}} \leq x \log x \leq (x-1) + (x-1)^2.
		\end{displaymath}
		The lemma can be proved by taking $x=f_i$ in the above inequality and summing up all $1 \leq i \leq N$.
	\end{proof}
\end{lemma}

A straightforward corollary of the above lemma is given as follows.

\begin{lemma} \label{lm:horder}
	For $\boldf^{(1)} \in \mathbb{R}_+^N$ and $\boldf^{(2)} \in \mathbb{R}_+^N$ with $\| \boldf^{(1)} \|_1 = \| \boldf^{(2)} \|_1=V$, if $\eta(\boldf^{(1)}) \leq \eta(\boldf^{(2)})$, then it holds that
	
	\begin{displaymath}
		\| \boldf^{(1)} - \boldone \|_2^2 \leq 2 \| \boldf^{(1)} \|_{\infty} \| \boldf^{(2)} -\boldone \|_2^2.
	\end{displaymath}	
\end{lemma}

After showing the equivalence between entropy function and 2-norm, we will 
%introduce one lemma and its corollaries to 
proceed to discuss the relationship between $\eta(\boldf^{(1)}) - \eta(\boldf^{(2)})$ and $\| \boldf^{(1)} - \boldf^{(2)} \|_2$ for any two vectors $\boldf^{(1)}$ and $\boldf^{(2)}$. By the definition of $\eta(\cdot)$, we are inspired to study the estimation of $h(x) - h(y)$. The result is presented in the following lemma.

\begin{lemma}\label{lm:smallhdiff}
    Given $0 < C_0 \leq 1$, $y \geq 0$ and $x \geq C_0$, if $y \geq C_0$ or $h(x) > h(y)$, then 
	\begin{equation}
	\label{eq:hdiff}
		\left| h(x) - h(y) \right| \leq \max \left( 2, 2 | \log ( C_0 ) | \right) \left| x - y \right| \left( \left| x - 1 \right| + \left| y - 1 \right| \right).
	\end{equation}
\end{lemma}

\begin{proof}
    If $x=y$, it is obvious that the lemma is correct. It remains to prove the lemma when $x \neq y$.
    
    By the mean value theorem,
	\begin{equation} \label{eq:hmean}
		h(x) -  h (y) = \log (\xi) ( x - y ),
	\end{equation}
	where $\xi$ is between $x$ and $y$. If $\log (\xi) \geq 0$, it holds that $\xi \geq  1$ and \begin{displaymath}
	    |  \log (\xi) | = \log (\xi) \leq \xi -1 \leq \max (x -1, y-1) \leq |x-1| + |y-1|.
	\end{displaymath}
	Therefore, if $\log (\xi) \geq 0$, \cref{eq:hmean} becomes
	\begin{equation} \label{eq:coff1}
	    | h(x) -  h (y) | \leq | x - y | \left( |x-1| + |y-1| \right). 
	\end{equation}
	
	Next we assume $h(x) > h(y)$. If $h(x) > h(y)$ and $x>y$, \cref{eq:hmean} implies $\log (\xi) >0$, which gives \cref{eq:coff1}. If $h(x) > h(y)$ and $x<y$, \cref{eq:hmean} implies $\log (\xi) <0$ and $\xi \leq 1$. In this case, $y > x \geq C_0$, which implies $\xi \geq C_0$ and $\log (\xi) \geq \log(C_0)$. Therefore, \cref{eq:hmean} becomes
	\begin{equation} \label{eq:logmean}
	    | h(x) -  h (y) | = - \log (\xi)  | x - y | \leq - \log(C_0) | x - y | = |\log(C_0) | | x - y |.
	\end{equation}
	On the other hand, by the mean value theorem,
	\begin{displaymath}
	    - \log (\xi)  = \log(1) - \log (\xi) = \frac{1}{\xi_2} (1  - \xi) \leq \frac{1}{C_0} \left( |x-1| + |y-1| \right),
	\end{displaymath}
	where $\xi_2 \in [\xi, 1] \subset [C_0, 1]$. %The last inequality comes from $\xi_2 \geq \xi \geq C_0$ and $(1  - \xi) \leq (|x-1| + |y-1|)$. Combing with \cref{eq:logmean}, it holds that
	The above results can be summarized into the following estimation:
	\begin{equation} \label{eq:logc0}
	    | h(x) -  h (y) | \leq | x - y |  \min \left( | \log (C_0) | , \frac{1}{C_0} \left( |x-1| + |y-1| \right) \right).
	\end{equation}
	If we further assume $x \geq 1/2$ and $y \geq 1/2$, then \cref{eq:logc0} is satisfied with $C_0 = 1/2$, which becomes
	\begin{displaymath}
	    | h(x) -  h (y) | \leq |x-y|\min \left( \log 2 , 2 \left( |x-1| + |y-1| \right) \right) \leq  2 | x - y| \left( |x-1| + |y-1| \right).
	\end{displaymath}
	Otherwise, if $x < 1/2$ or $y < 1/2$, we have $2 ( |x-1| + |y-1| ) \geq 1$. Therefore,
	\begin{displaymath}
	    \min \left( | \log (C_0) | , \frac{1}{C_0} \left( |x-1| + |y-1| \right) \right) \leq | \log (C_0) | \leq  2 | \log (C_0) | \left( |x-1| + |y-1| \right).
	\end{displaymath}
	Combining the two results above yields the inequality \eqref{eq:hdiff} when $h(x) > h(y)$.
	
	It remains only to consider the case $h(x) \leq h(y)$ and $y \geq C_0$. If $x<y$, \cref{eq:hmean} implies $\log (\xi) \geq 0$, which gives the result of \cref{eq:coff1}. Otherwise, $x>y$ implies $\log (\xi) \leq 0$. Since $x \geq C_0$ and $y \geq C_0$, it holds that $\xi \geq C_0$, and therefore $0 \geq \log (\xi) \geq \log(C_0)$, which also yields \cref{eq:logc0}. The rest of the proof is the same as the previous case.
%	
%	Now we have shown that if $y \geq C_0$ or $h(x) > h(y)$, either \cref{eq:coff1} or \cref{eq:logc2} holds. Since \cref{eq:coff1} implies \cref{eq:logc2}, we could conclude \cref{eq:logc2} holds, which proves the lemma.
\end{proof}

With the help of the above lemma, we could give an upper bound of the difference of entropy functions $\eta(\boldf^{(1)}) - \eta(\boldf^{(2)})$ in the following lemma.

\begin{lemma}\label{lm:hdiff}
	For $\boldf^{(1)} = (f^{(1)}_1, \ldots, f^{(1)}_N) \in \mathbb{R}_+^N$ and $\boldf^{(2)} = (f^{(2)}_1, \ldots, f^{(2)}_N) \in \mathbb{R}_+^N$ with $\| \boldf^{(1)} \|_1 = \| \boldf^{(2)} \|_1=V$, given $0 <C_0 \leq 1$, if $f^{(1)}_i \geq C_0$ for all $1 \leq i \leq N$ and $\boldf^{(2)}$ satisfies either of the following conditions: 
	\begin{enumerate}
	\item $f^{(2)}_i \geq C_0$ for all $1 \leq i \leq N$; \item $\eta(\boldf^{(2)}) < \eta(\boldf^{(1)})$; 
	\end{enumerate}
    then it holds that
	\begin{displaymath}
		|\eta(\boldf^{(1)}) - \eta(\boldf^{(2)})| \leq \max \left(2, 2 | \log (C_0) | \right) \| \boldf^{(1)} - \boldf^{(2)} \|_2 \left( \| \boldf^{(1)} - \boldone \|_2 + \| \boldf^{(2)} - \boldone \|_2\right).
	\end{displaymath}
\end{lemma}

\begin{proof}
    For simplicity, we use $M$ to denote the constant $\max \left(2, 2 | \log (C_0) | \right)$ in this proof. In the first case $f^{(2)}_i \geq C_0 >0$ for all $1 \leq i \leq N$, we can plug $x = f^{(1)}_i$ and $y = f^{(2)}_i$ in \cref{lm:smallhdiff} and sum over all $1 \leq i \leq N$. By using $\|\boldf^{(1)}\|_1 = \|\boldf^{(2)}\|_1$, we can obtain that
	\begin{displaymath}
	  |\eta(\boldf^{(1)}) - \eta(\boldf^{(2)})| \leq M \sum_{i=1}^N	\left(  |(f_i^{(1)} - 1)(f_i^{(1)} - f_i^{(2)})|
			+ |(f_i^{(2)} - 1)(f_i^{(1)} - f_i^{(2)})| \right) \Delta v_i.
	\end{displaymath}
	The lemma can be proven by the Cauchy-Schwarz inequality.
	
	In the second case $\eta(\boldf^{(2)}) < \eta(\boldf^{(1)})$, we have
	\begin{align*}
			& \eta(\boldf^{(1)}) - \eta(\boldf^{(2)})  \\
			={} & \sum_{h(f_i^{(1)}) \leq h(f_i^{(2)})} \left( h(f_i^{(1)})  - h(f_i^{(2)}) \right) \Delta v_i + \sum_{h(f_i^{(1)}) > h(f_i^{(2)})} \left( h(f_i^{(1)}) - h(f_i^{(2)}) \right) \Delta v_i \\
			\leq{} & \sum_{h(f_i^{(1)}) > h(f_i^{(2)})} \left(  h(f_i^{(1)}) - h(f_i^{(2)}) \right) \Delta v_i \\
			\leq{} & M \sum_{h(f_i^{(1)}) > h(f_i^{(2)})} \left(  |(f_i^{(1)} - 1)(f_i^{(1)} - f_i^{(2)})| + |(f_i^{(2)} - 1)(f_i^{(1)} - f_i^{(2)})| \right) \Delta v_i,
	\end{align*}
	where the last inequality is again the result of \cref{lm:smallhdiff}. The lemma naturally follows by extending the range of summation of $i$ to $1,\dots,N$ and applying the Cauchy-Schwarz inequality.
\end{proof}

In the proof of case 2, we applied \cref{lm:smallhdiff} only to $f_i^{(1)}$ and $f_i^{(2)}$ with $h(f_i^{(1)}) > h(f_i^{(2)})$. This allows us to relax the condition ``$f^{(1)}_i \geq C_0$ for all $1 \leq i \leq N$'' in the case $\eta(\boldf^{(1)}) > \eta(\boldf^{(2)})$. In fact, we need $f^{(1)}_i > C_0$ only for the components that require \cref{lm:smallhdiff}. We write this result in the following corollary:

\begin{corollary} \label{lm5col1}
 	For $\boldf^{(1)} = (f^{(1)}_1, \dots, f^{(1)}_N) \in \mathbb{R}_+^N$ and $\boldf^{(2)} = (f^{(2)}_1, \dots, f^{(2)}_N) \in \mathbb{R}_+^N$ with $\| \boldf^{(1)} \|_1 = \| \boldf^{(2)} \|_1=V$, we assume $\eta(\boldf^{(1)}) > \eta(\boldf^{(2)})$. If there exists $0 < C_0 <1$ such that for any $i = 1,\dots,N$, either $f^{(1)}_i \geq C_0$ or $h(f_i^{(1)}) \leq h(f_i^{(2)})$ is satisfied, then it holds that 
 	\begin{displaymath}
 		|\eta(\boldf^{(1)}) - \eta(\boldf^{(2)})| \leq \max \left(2, 2 | \log (C_0) | \right) \| \boldf^{(1)} - \boldf^{(2)} \|_2 \left( \| \boldf^{(1)} - \boldone \|_2 + \| \boldf^{(2)} - \boldone \|_2\right).
 	\end{displaymath}
\end{corollary}

We are now ready to prove \cref{thm1}.

%\subsection{Proof of \cref{thm1}} This subsection provides the proof of \cref{thm1}.
\begin{proof}[Proof of \cref{thm1}]
	The convexity of $\eta(\cdot)$ implies
	\begin{displaymath}
		\eta(\boldf^{n+1}+\beta(\boldone - \boldf^{n+1})) \leq \beta \eta(\boldone) + (1 - \beta) \eta(\boldf^{n+1}).
	\end{displaymath}
	By \cref{eq:enfix} with $\eta(\boldone)=0$, the above inequality is equivalent as
	\begin{equation} \label{thm1:ratio}
		\beta \leq \frac{\eta(\boldf^{n+1}) - \eta( \boldf(t_{n+1}) )}{\eta(\boldf^{n+1})}.
	\end{equation}
	The numerator in \cref{thm1:ratio} can be estimated by
	%\begin{displaymath}
	%	\begin{split}
	\begin{align*}
			& \eta(\boldf^{n+1}) - \eta( \boldf(t_{n+1}) ) \\
		    \leq{} & M \| \boldf^{n+1} - \boldf(t_{n+1}) \|_2 \left( \| \boldf^{n+1} - \boldone \|_2 + \| \boldf(t_{n+1}) - \boldone \|_2\right) & \text{(\cref{lm:hdiff})}  \\
			\leq{} & M \| \boldf^{n+1} - \boldf(t_{n+1}) \|_2 \left( 1 + \sqrt{ \| \boldf(t_{n+1}) \|_{\infty} } \right) \| \boldf^{n+1} - \boldone \|_2, & \text{(\cref{lm:horder})}
	\end{align*}
	where $M = \max(2, 2|\log(C_0)|)$.
	%	\end{split}
	%\end{displaymath}
	On the other hand, according to \cref{lm:h=twonorm}, the denominator in \cref{thm1:ratio} satisfies
	\begin{displaymath}
		\eta(\boldf^{n+1}) \geq \frac{1}{2\| \boldf^{n+1} \|_{\infty}} \| \boldf^{n+1} - \boldone \|_2^2.
	\end{displaymath}
	Therefore,
	\begin{displaymath}
		\| \beta (\boldone - \boldf^{n+1}) \|_2 \leq 2M \|\boldf^{n+1}\|_{\infty} (1 + \sqrt{\| \boldf(t_{n+1}) \|_{\infty}}) \| \boldf^{n+1} - \boldf(t_{n+1}) \|_2.
	\end{displaymath}
\end{proof}

In this case, we would like to give a remark on the practical choice of $\beta_p$ in \cref{eq:enfixpratic}. Instead of solving $\eta(\boldf^{n+1} + \beta_p( \boldone - \boldf^{n+1})) = \eta(\boldf^n)$, we can simply take $\hat{\beta}_p = ({\eta(\boldf^{n+1}) - \eta( \boldf^n ) })/{ \eta( \boldf^{n+1} ) }$, which equals the upper bound in \cref{thm1:ratio}. Note that the convexity of function $\eta(\cdot)$ implies $\eta(\boldf^n) = (1 - \hat{\beta}_p)\eta(\boldf^{n+1}) + \hat{\beta}_p \eta(\boldone) \geq \eta(\boldf^{n+1} + \hat{\beta}_p (\boldone - \boldf^{n+1}))$. Therefore, under the condition of \cref{thm1}, if we change the numerical solution at $(n+1)$th step to $\boldf^{n+1} + \hat{\beta}_p (\boldone - \boldf^{n+1})$, it still holds that $\| \hat{\beta}_p  ( \boldone - \boldf^{n+1})\|_2 \leq M \| \boldf(t_{n+1}) - \boldf^{n+1} \|_2$.
%This approach can also be applied to \cref{thmlog}, which will be proven in the following subsection.

\subsection{Proof of \cref{thmlog}}
Different from the previous proof, in \cref{thmlog}, we allow the solution to have components arbitrarily close to zero, so that \cref{lm:hdiff} cannot be directly applied.
To overcome this difficulty, we introduce a regularization term before using \cref{lm:hdiff}. The details are given as follows.

\begin{proof}[Proof of \cref{thmlog}]
    %This proof consists of two parts. In the first part, we make a perturbation as regularization, after which all components are greater or equal to $\| \boldf^{n+1} - \boldf(t_{n+1}) \|_{2}$. In the second part, we follow the proof of \cref{thm1}. 
    
    For simplicity, we let $\varepsilon = \| \boldf^{n+1} - \boldf(t_{n+1}) \|_{2}$.
    To avoid dealing with zero components, we first regularize the numerical solution $\boldf^{n+1}$ by
    \begin{equation} \label{eq:perturb}
    \boldf^{n+1,1} = \boldf^{n+1} + \varepsilon (\boldone - \boldf^{n+1}), 
    \end{equation}
    after which $f^{n+1,1}_i \geq \varepsilon$ for all $i = 1,\dots,N$. On the other hand, since $\| \boldone - \boldf^{n+1} \|_{\infty} \leq \max( 1, \| \boldf \|_{\infty} -1) \leq \| \boldf \|_{\infty}$, the $L^2$ norm of the perturbation introduced by the regularization satisfies
    \begin{displaymath}
    \|\varepsilon (\boldone - \boldf^{n+1}) \|_2 \leq \sqrt{V} \| \boldf^{n+1} \|_{\infty} \varepsilon.
    \end{displaymath}
    
    After perturbation, if $\eta(\boldf^{n+1,1}) < \eta(\boldf(t_{n+1}))$, then we have $\beta < \varepsilon$ so that the conclusion of the theorem is drawn. If $\eta(\boldf^{n+1,1}) > \eta(\boldf(t_{n+1}))$, we can find $\beta_2 \in (0,1]$ such that
    \begin{displaymath}\eta(\boldf^{n+1,1}+\beta_2(\boldone - \boldf^{n+1,1})) = \eta(\boldf(t_{n+1})),
    \end{displaymath}
    which is identical to \cref{eq:enfix} by replacing $\boldf^{n+1}$ to $\boldf^{n+1,1}$. Therefore, we can set $C_0 = \varepsilon$ in \cref{thm1} to obtain
% 	\begin{displaymath}
% 		\beta_2 \leq \frac{H[\boldf^{n+1,1}] - H[ \boldf(t_{n+1}) ]}{H[\boldf^{n+1,1}]},
% 	\end{displaymath}
% 	which is similar to \cref{thm1:ratio}. Following the proof after \cref{thm1:ratio} and replacing \cref{lm5col2} by \cref{lm5col3}, we could get
	\begin{displaymath}
		\| \beta_2 (\boldone - \boldf^{n+1,1}) \|_2 \leq M_1 \| \boldf^{n+1,1} - \boldf(t_{n+1}) \|_2,
	\end{displaymath}
	and by the proof of \cref{thm1}, we know that
	\begin{displaymath}
	\begin{aligned}
	M_1 &= 4 \max \left(1, \left| \log \varepsilon \right| \right) \|\boldf^{n+1,1}\|_{\infty} (1 + \sqrt{\| \boldf(t_{n+1}) \|_{\infty}}) \\
	&\leq 4 ( 1 + \left| \log \varepsilon \right| ) \|\boldf^{n+1}\|_{\infty} (1 + \sqrt{\| \boldf(t_{n+1}) \|_{\infty}}),
	\end{aligned}
	\end{displaymath}
	since $\| \boldf^{n+1,1} \|_{\infty} \leq \|\boldf^{n+1}\|_{\infty}$.
	%and $8 \leq 3 \left| \log \varepsilon \right|$ for $\varepsilon \leq 1/2$.
	
	If we define
	\begin{equation} \label{eq:afterperturb}
	    \hat{\boldf}^{n+1} = \boldf^{n+1,1}+\beta_2(\boldone - \boldf^{n+1,1}) =
		\boldf^{n+1} + (\varepsilon + \beta_2 - \varepsilon \beta_2) (\boldone - \boldf^{n+1}),
	\end{equation}
	then by $\eta(\hat{\boldf}^{n+1}) = \eta(\boldf(t_{n+1}))$ we know that $\beta = \varepsilon + \beta_2 - \varepsilon \beta_2$. Thus
	it holds that
	\begin{align*}
	    	\| \beta(\boldone - \boldf^{n+1}) \|_2 &= \| \hat{\boldf}^{n+1} - \boldf^{n+1} \|_2 \leq  \| \hat{\boldf}^{n+1} - \boldf^{n+1,1} \|_2 + \| \boldf^{n+1,1} - \boldf^{n+1} \|_2 \\
			& \leq M_1 \| \boldf^{n+1,1} - \boldf(t_{n+1}) \|_2 + \| \boldf^{n+1,1} - \boldf^{n+1} \|_2 \\
			& \leq M_1 (\| \boldf^{n+1,1} - \boldf^{n+1} \|_2 + \varepsilon) + \| \boldf^{n+1,1} - \boldf^{n+1} \|_2 \\
			& \leq M_1 (\sqrt{V} \| \boldf^{n+1} \|_{\infty} \varepsilon + \varepsilon) + \sqrt{V} \| \boldf^{n+1} \|_{\infty} \varepsilon \leq M_2 \varepsilon (|\log \varepsilon| + 1),
	\end{align*}
	where $M_2 = 8 (\sqrt{V} \| \boldf^{n+1} \|_{\infty} + 1) \|\boldf^{n+1}\|_{\infty} (1 + \sqrt{\| \boldf(t_{n+1}) \|_{\infty}})$.
%	
	%On the other hand, if $H[\boldf^{n+1,1}] \leq H[\boldf(t_{n+1})]$, we take $\beta_2=0$ and construct $\hat{\boldf}^{n+1}$ as \cref{eq:afterperturb}. Therefore, $ \| \hat{\boldf}^{n+1} - \boldf^{n+1,1} \|_2 = 0$ and the above inequality still holds. Furthermore, in this case, $H[\hat{\boldf}^{n+1}] = H[\boldf^{n+1,1}] \leq H[\boldf(t_{n+1})]$.
%	
%	Plugging \cref{eq:perturb} into \cref{eq:afterperturb}, it holds that
%	\begin{displaymath}
%		\hat{\boldf}^{n+1} = \boldf^{n+1} + (\varepsilon + \beta_2 - \varepsilon \beta_2) (\boldone - \boldf^{n+1}).
%	\end{displaymath}
%	Since $H[\hat{\boldf}^{n+1}] \leq H[\boldf(t_{n+1})]$, the monotonicity of $H[\boldf^{n+1} + \omega ( \boldone - \boldf^{n+1})]$ with respect to $\omega$ implies the $\beta$ given in \cref{eq:enfix} satisfies $\beta \leq (\varepsilon + \beta_2 - \varepsilon \beta_2)$. Therefore,
%	\begin{displaymath}
%		\| \beta  ( \boldone - \boldf^{n+1})\|_2 \leq \| (\varepsilon + \beta_2 - \varepsilon \beta_2) ( \boldone - \boldf^{n+1})\|_2  = \| \hat{\boldf}^{n+1} - \boldf^{n+1} \|_2.
%	\end{displaymath}
\end{proof}

%The main idea of the proof of \cref{thmlog} can be regarded as a two-step procedure. In the first step, a regularization is applied to avoid the singularity when components is closed to $0$. Moreover, this regularization do not break the consistency. Then in the second step, we handle the regularized solution. This idea will be used in the proof of \cref{thm2}.
The proof of this theorem follows the two-step procedure, which will also be applied in the proof of \cref{thm2}.

\subsection{Proof of \cref{thm2}}
To prove \cref{thm2}, we deal with the components with $f_i^{n+1} < \frac{2}{3}$ and $f_i^{n+1} > \frac{2}{3}$ separately. The difference between these two cases can be seen from the following lemma:

\begin{lemma} \label{lm:infperturb}
	For $\boldf^{(1)} \in \mathbb{R}_+^N$ and $\boldf^{(2)} \in \mathbb{R}_+^N$ with $\| \boldf^{(1)} -  \boldf^{(2)} \|_{\infty} \leq \frac{1}{3}$, define
	\begin{displaymath}
		\boldf^{(3)} = \boldf^{(1)} + \beta_1 (\boldone - \boldf^{(1)}),
	\end{displaymath}
	where $\beta_1 = 3 \| \boldf^{(1)} - \boldf^{(2)} \|_{\infty}$. If $\| \boldf^{(1)} \|_1 = V$, then $\boldf^{(3)}$ satisfies following properties:
	\begin{enumerate}
		\item For all $k$ such that $f_k^{(1)} < \frac{2}{3}$, it holds that $h(f^{(3)}_k) \leq h(f^{(2)}_k)$;
		\item For all $k$ such that $f_k^{(1)} \geq \frac{2}{3}$, it holds that $f^{(3)}_k \geq \frac{2}{3}$;
		\item $\| \boldf^{(3)} - \boldf^{(1)} \|_{\infty} \leq 3 \| \boldf^{(1)} \|_{\infty} \| \boldf^{(1)} - \boldf^{(2)} \|_{\infty}$.
	\end{enumerate}
\end{lemma}

\begin{proof}
%	From the notation of $\boldf^{(3)}$, we know for $1 \leq k \leq N$,
%	\begin{displaymath}
%		f_k^{(3)} - f_k^{(1)} = \beta_1 (1 - f_k^{(1)}).
%	\end{displaymath}
	For those $k$ such that $f_k^{(1)} < \frac{2}{3}$, we have $1 - f_k^{(1)} \geq \frac{1}{3}$. Thus %Together with $\beta_1 = 3 \| \boldf^{(1)} - \boldf^{(2)} \|_{\infty} \geq 3|f^{(1)}_k - f^{(2)}_k|$,
	\begin{displaymath}
		f_k^{(3)} - f_k^{(1)} = \beta_1 (1 - f_k^{(1)}) \geq 3|f^{(1)}_k - f^{(2)}_k| \cdot (1 - f_k^{(1)}) \geq |f^{(1)}_k - f^{(2)}_k| \geq f^{(2)}_k - f^{(1)}_k,
	\end{displaymath}
	which yields $f_k^{(3)} \geq f_k^{(2)}$. Since $f_k^{(3)}$ is the convex combination of $1$ and $f_k^{(1)}$, we have $0 \leq f_k^{(3)} \leq 1$. Since $h(\cdot)$ is monotonically decreasing on $[0,1]$, we conclude that $h(f^{(3)}_k) \leq h(f^{(2)}_k)$.
	
	The second property is obvious since $f_k^{(3)}$ lies between $f_k^{(1)}$ and $1$.
	
	As for the third property, it should be noted that $\| \boldf^{(1)} \|_1 = V$ implies $\|\boldf^{(1)}\|_{\infty} \geq 1$. Therefore,
	%\begin{displaymath}
	%	\begin{split}
	\begin{align*}
			\| \boldf^{(3)} - \boldf^{(1)} \|_{\infty} & = \beta_1 \| \boldone - \boldf^{(1)} \|_{\infty} \\
			& \leq \max(1, \| \boldf^{(1)} \|_{\infty}-1) \beta_1
			\leq 3 \| \boldf^{(1)} \|_{\infty} \| \boldf^{(1)} - \boldf^{(2)} \|_{\infty}.
	\end{align*}
	%	\end{split}
	%\end{displaymath}
\end{proof}

The first property in \cref{lm:infperturb} shows how we deal with the small components, and this only holds when $\beta_1$ is proportional to the difference between $\boldf^{(1)}$ and $\boldf^{(2)}$ measured by the infinity norm, leading to the form of the right-hand side in the conclusion of \cref{thm2}. For the remaining terms, an $O(1)$ lower bound exists, so that the same technique as \cref{thm1} can be applied. The details of the proof are given below:

\begin{proof}[Proof of \cref{thm2}]
	By \cref{lm:infperturb}, we could pick $\beta_1 = 3 \| \boldf^{n+1} - \boldf(t_{n+1})\|_{\infty}$ and construct
	\begin{equation} \label{eq:thm2f1}
		\boldf^{n+1,1} = \boldf^{n+1} + \beta_1 (\boldone - \boldf^{n+1}),
	\end{equation}
	If $\eta(\boldf^{n+1,1}) \leq \eta(\boldf(t_{n+1}))$, the proof is already completed. If $\eta(\boldf^{n+1,1}) > \eta(\boldf(t_{n+1}))$, we construct $\hat{\boldf}^{n+1}$ as \cref{eq:afterperturb} such that $\eta(\hat{\boldf}^{n+1}) = \eta(\boldf(t_{n+1}))$, and thus $\beta = \beta_1 + \beta_2 - \beta_1 \beta_2$.
	According to \cref{lm:infperturb},
	those components $i$ where $h(f_i^{n+1,1}) > h(f_i(t_{n+1}))$ satisfy $f_i^{n+1,1} \geq \frac{2}{3}$. Therefore, \cref{lm5col1} could be applied with $C_0 = \frac{2}{3}$, and we could mimic the proof of \cref{thm1} with only replacement from \cref{lm:hdiff} to \cref{lm5col1} in the proof. As a result, 
%	\begin{equation} \label{eq:thm2f2}
%		\hat{\boldf}^{n+1} = \boldf^{n+1,1} + \beta_2 ( \boldone - \boldf^{n+1,1}),
%	\end{equation}
	by the conclusion of \cref{thm1}, it holds that
	\begin{displaymath}
		\| \beta_2  ( \boldone - \boldf^{n+1,1})\|_2 \leq  M_1 \| \boldf(t_{n+1}) - \boldf^{n+1,1} \|_{2},
	\end{displaymath}
	where $M_1 =4 \max \left(1, | \log (C_0) | \right) \|\boldf^{n+1,1}\|_{\infty} (1 + \sqrt{\| \boldf(t_{n+1}) \|_{\infty} } )$ taken from the proof of \cref{thm1}. Moreover, $\boldf^{n+1,1}$ in $M_1$ could be replaced by $\boldf^{n+1}$ since $\| \boldf^{n+1,1}\|_{\infty} \leq \| \boldf^{n+1} \|_{\infty} $. Then, similar to the second step in the proof of \cref{thmlog}, it holds that
	%\begin{displaymath}
	%	\begin{split}
	\begin{align*}
			\| \hat{\boldf}^{n+1} - \boldf^{n+1} \|_2 & \leq  \| \hat{\boldf}^{n+1} - \boldf^{n+1,1} \|_2 + \| \boldf^{n+1,1} - \boldf^{n+1} \|_2 \\
			& \leq M_1 \| \boldf(t_{n+1}) - \boldf^{n+1,1} \|_{2} + \| \boldf^{n+1,1} - \boldf^{n+1} \|_2 \\
			& \leq M_1 \sqrt{V} \| \boldf(t_{n+1}) - \boldf^{n+1,1} \|_{\infty} + \sqrt{V} \| \boldf^{n+1,1} - \boldf^{n+1} \|_{\infty} \\
			& \leq M_1 \sqrt{V} \| \boldf(t_{n+1}) - \boldf^{n+1} \|_{\infty} + \sqrt{V} (M_1 + 1) \| \boldf^{n+1,1} - \boldf^{n+1} \|_{\infty} \\
			& \leq \left(  M_1  \sqrt{V} + 3 \sqrt{V} (M_1 + 1)  \| \boldf^{n+1} \|_{\infty}  \right)  \| \boldf(t_{n+1}) - \boldf^{n+1} \|_{\infty},
	\end{align*}
	where the last ``$\leq$'' is the result of \cref{lm:infperturb}.
	This completes the proof since $\| \beta  ( \boldone - \boldf^{n+1})\|_2 = \| \hat{\boldf}^{n+1} - \boldf^{n+1} \|_2$.
\end{proof}

\subsection{Proof of \cref{thm3}}
In this subsection, we will prove \cref{thm3}. Before that, we would like to introduce two lemmas. \cref{lm: opt} comes from optimization, which illustrates the infinity norm of optimal solution could be bounded by the $L^2$ norm of it. Based on \cref{lm: opt}, we make a decomposition of the (relative) entropy function in \cref{eq:decomH} and then introduce \cref{lemma:h2h1} to estimate the difference of decomposed entropy functions.

As assumed in the theorem, we suppose all the components of $\boldf$ are sorted in the ascending order:
\begin{displaymath}
	f_1 \leq f_2 \leq \cdots \leq f_N = \| \boldf \|_{\infty}.
\end{displaymath}
Note that this does not affect the definition of entropy and the numerical scheme for the entropy fix.

\begin{lemma} \label{lm: opt}
	For any $C_1, C_f \in (0,1]$ and positive integer $N$, let $I_1 = \min \{ I \mid \sum_{i=1}^I \Delta v_i \geq C_1 V \}$. If $\boldf \in \mathbb{R}_+^N$ satisfies
	\begin{displaymath}
	    f_i \leq 1/2 \text{ for all } i = 1,\ldots,I_1 \qquad \text{and} \qquad \frac{1}{|\log f_1|} \geq \frac{C_f}{|\log f_{I_1}|},
	\end{displaymath}
	then when $\varepsilon < \frac{1}{2} \sqrt{C_1 V}$, the solution $\boldg^* = (g^*_1, \ldots, g^*_{I_1})^T \in \mathbb{R}^{I_1}$ of the following optimization problem
	\begin{equation} \label{eq:opt}
		\argmin_{g_1, \ldots, g_{I_1}} \sum_{i=1}^{I_1} h(f_i + g_i) \Delta v_i, \qquad \text{s.t. } \sum_{i=1}^{I_1} g_i^2 \Delta v_i \leq \varepsilon^2,
	\end{equation}
	satisfies $0 \le g_{I_1}^* \le \cdots \le g_1^* \le (\sqrt{C_1 V} C_f)^{-1} \varepsilon$ and $C_f \leq g_{I_1}^* / g_1^* \leq 1$.
\end{lemma}

\begin{proof}
	The proof utilizes the Karush–Kuhn–Tucker (KKT) sufficient conditions for optimization problems \cite[Chapter 3.5]{nonlinearopt2006}. It is easy to verify that both the objective function and the constraint are continuously differentiable convex functions with respect to $(g_1, \ldots, g_{I_1})^T$. Therefore, if the following conditions hold for $\lambda^* \in \mathbb{R}$ and $\boldg^* = (g_1^*, \ldots, g_{I_1}^*)^T$,
	\begin{equation} \label{kkteq}
		\left\{ 
		\begin{aligned}
			& h'(f_i  + g_i^*) + 2 \lambda^* g_i^* = 0, \qquad \forall \  1 \leq i \leq I_1, \\
			& \sum_{i=1}^{I_1} (g_i^*)^2 \Delta v_i \leq \varepsilon^2, \\
			& \lambda^* \geq 0, \\
			& \lambda^* (\sum_{i=1}^{I_1} (g_i^*)^2 \Delta v_i - \varepsilon^2) = 0,
		\end{aligned}
		\right.
	\end{equation}
	then $\boldg^*$ is the global minimum of the optimization problem.
	
	First, we claim that $\lambda^* \neq 0$, so that
	\begin{equation} \label{eqconstraint}
	    \sum_{i=1}^{I_1} (g_i^*)^2 \Delta v_i = \varepsilon^2
	\end{equation}
	due to the last equation in \cref{kkteq}. If $\lambda^*$ equals $0$, then $h'(f_i + g_i^*) = 0$, which yields $g_i^* = 1 - f_i \geq \frac{1}{2}$. Therefore,
	\begin{displaymath}
		\sum_{i=1}^{I_1} (g_i^*)^2 \Delta v_i \geq \frac{\sum_{i=1}^{I_1} \Delta v_i }{4} \geq \frac{C_1 V}{4} > \varepsilon^2,
	\end{displaymath}
	which contradicts with the second inequality in \cref{kkteq}.
	
	Now we would like to establish the existence and uniqueness of the solution. We first focus on the first equation in \cref{kkteq}. For any $1 \leq i \leq I_1$ and fixed $\lambda^*>0$, there exist one unique $g_i^* \in (0,1)$ satisfying $h'(f_i  + g_i^*) + 2 \lambda^* g_i^* = 0$. This is because the function $\zeta_i(x) := h'(f_i + x) + 2\lambda^* x$ is monotonically increasing, and
	\begin{displaymath}
	    \zeta(0) = \log f_i \leq \log (\frac{1}{2}) < 0, \qquad \zeta(1) \geq 2 \lambda^* >0.
	\end{displaymath}
	Thus it remains to demonstrate that $\lambda^*$ is unique. Inspired by the first equation in \cref{kkteq}, we define
	\begin{displaymath}
	    \sigma(x) = -\frac{h'(f_i + x)}{2x} = -\frac{\log(f_i + x)}{2x}, \qquad x \in (0,1-f_i].
	\end{displaymath}
	Then its inverse function $\sigma_i^{-1}(y)$ satisfies
	\begin{equation} \label{eq:eta_inverse}
	    y = -\frac{\log(f_i + \sigma_i^{-1}(y))}{2\sigma_i^{-1}(y)} \quad \text{and} \quad
	    \sigma_i^{-1}(y) = \frac{W_0(2y e^{2 y f_i}) - 2 y f_i}{2y},
	\end{equation}
	where $W_0(\cdot)$ is the Lambert $W$ function \cite{Corless1996lambertw} satisfying $W_0(x) e^{W_0(x)} = x$. For $\sigma_i(\cdot)$ and $\sigma_i^{-1}(\cdot)$, we have the following properties:
	\begin{enumerate}
    	\item $\sigma_i(x)$ is monotonically decreasing, so is $\sigma_i^{-1}(x)$ (this requires $f_i \leq \frac{1}{2}$);
    	\item $\sigma_i(g_i^*) = \lambda^*$ and $g_i^* = \sigma_i^{-1}(\lambda^*)$;
    	\item $\sigma_i^{-1}(0) = 1 - f_i \geq \frac{1}{2}$ and $\sigma_i^{-1}(y) \rightarrow 0$ as $y \rightarrow +\infty$.
	\end{enumerate}
	Here the limit of $\sigma_i^{-1}(y)$ at $+\infty$ can be obtained by the inequality (see \cite{Hoorfar2008inequalities})
	\begin{displaymath}
	     W_0(x) \leq \log(x) - \log (\log(x)) + \frac{e}{e-1} \frac{\log (\log(x))}{\log(x)}, \qquad \forall x \geq e.
	\end{displaymath}
	Furthermore, if we define
	\begin{displaymath}
	    \Xi(y) = \sum_{i=1}^{I_1} [\sigma_i^{-1}(y)]^2 \Delta v_i, \qquad y \in [0,+\infty),
	\end{displaymath}
	then by the three properties of $\sigma_i$, we have
	\begin{enumerate}
	    \item $\Xi(y)$ is a decreasing function since each $\sigma_i^{-1}(y)$ is monotonically decreasing;
	    \item $\Xi(\lambda^*) = \varepsilon^2$ according to \cref{eqconstraint};
	    \item $\Xi(0) \geq \frac{1}{4} \sum_{i=1}^{I_1} \Delta v_i > \varepsilon^2$, and $\Xi(y) \rightarrow 0$ as $y \rightarrow +\infty$.
	\end{enumerate}
	These properties show the existence and uniqueness of $\lambda^*$.
	
	Next, we will show $g_{I_1}^* \leq \cdots \leq g_1^*$. For any $1 \leq i \leq j \leq I_1$, $f_i \leq f_j$ implies
	\begin{displaymath}
	    \sigma_j(g_j^*) = \lambda^* = \sigma_i(g_i^*) \geq \sigma_j(g_i^*).
	\end{displaymath}
	Using the fact that $\sigma_j(\cdot)$ is decreasing, we see that $g_j^* \leq g_i^*$. To get the bound of $g_1^*$, we need the following two results:
	\begin{itemize}
	    \item By \cref{eq:eta_inverse}, we have
    	\begin{displaymath}
    	    \lim_{y \rightarrow +\infty} \frac{\sigma_i^{-1}(y)}{\sigma_1^{-1}(y)} = 
    	    \lim_{y \rightarrow +\infty} \frac{\log(f_i + \sigma_i^{-1}(y))}{\log(f_1 + \sigma_1^{-1}(y))} = \frac{\log(f_i)}{\log(f_1)} \geq C_f;
    	\end{displaymath}
    	\item By straightforward calculation, we have
    	\begin{displaymath}
            \frac{\mathrm{d}}{\mathrm{d}y} \left(\frac{\sigma_i^{-1}(y)}{\sigma_1^{-1}(y)} \right) = \frac{W_0(2y e^{2 y f_1}) - W_0(2y e^{2 y f_i})}{y(1 + W_0(2y e^{2 y f_1}))(1 + W_0(2y e^{2 y f_i}))}
            \frac{\sigma_i^{-1}(y)}{\sigma_1^{-1}(y)} \leq 0.
    	\end{displaymath}
	\end{itemize}
	These results indicate that
	\begin{displaymath}
	    \frac{g_i^*}{g_1^*} = \frac{\sigma_i^{-1}(\lambda^*)}{\sigma_1^{-1}(\lambda^*)} \geq C_f,
	\end{displaymath}
	and thus
	\begin{displaymath}
	    g_1^* = \varepsilon \left( \sum_{i=1}^{I_1} \frac{(g_i^*)^2}{(g_1^*)^2} \Delta v_i\right)^{-1/2} \leq
	    \varepsilon \left( \sum_{i=1}^{I_1} C_f^2 \Delta v_i \right)^{-1/2} \leq \frac{\varepsilon}{\sqrt{C_1 V} C_f}.
	\end{displaymath}
	This completes the proof.
\end{proof}

One corollary of the above lemma is the extension to a continuous version, with identical optimal solution $\boldg^*$ in the sense of piesewise constant function. For the ease of this extension, we would like to introduce the (partial) sum of first $i$ parameters $\Delta v_i$ as
\begin{equation} \label{eq: partialSum}
   S_0 = 0, \qquad S_i = \sum_{j=1}^i \Delta v_j, \qquad i = 1, \dots, N.
\end{equation}
Then we have the following lemma.
\begin{corollary} \label{lm: optcont}
    Under the condition of \cref{lm: opt}, if a piesewise constant function defined on $(0, S_{I_1}]$ is introduced as
    \begin{displaymath}
       f(v) = f_i, \qquad v \in (S_{i-1}, S_i], \quad i = 1, \dots, I_1,
    \end{displaymath}
    then the solution $g^*(v) \in L^2((0, S_{I_1}])$ of the following optimization problem
    \begin{equation} \label{eq: contOptproblem}
       \argmin_{g \in L^2((0, S_{I_1}])} \int_{0}^{S_{I_1}} h(f(v) + g(v)) \mathrm{d}v, \qquad \text{s.t. } \| g\|_2^2 :=\int_{0}^{S_{I_1}} (g(v))^2 \mathrm{d}v \leq \varepsilon^2,
    \end{equation}
    is equal to a piecewise constant function \text{a.e.} as
    \begin{displaymath}
       g^*(v) = g^*_i, \qquad v \in (S_{i-1}, S_i], \quad i = 1, \dots, I_1,
    \end{displaymath}
    where $g^*_i$ is the component of the optimal solution $\boldg^*$ in \cref{lm: opt}.
\end{corollary}

\begin{proof}
To prove the corollary, it suffices to show that for every $i = 1,\cdots,I_1$, the function $g^*(v)$ is a constant on $(S_{i-1}, S_i]$ except for a set with measure zero, so that the optimization problem \cref{eq: contOptproblem} is essentially equivalent to \cref{eq:opt}. Suppose that $g^*(v)$ is essentially not a constant on $(S_{i-1}, S_i]$ for some $i$. We define the function $\hat{g}(v)$ by
\begin{displaymath}
\hat{g}(v) = \left\{ \begin{array}{ll}
  \frac{1}{\Delta v_i} \int_{S_{i-1}}^{S_i} g^*(v) \,\mathrm{d} v, & \text{if } v \in (S_{i-1}, S_i], \\
  g^*(v), & \text{otherwise}.
\end{array} \right.
\end{displaymath}
By H\"{o}lder's inequality (on $(S_{i-1}, S_i]$), it is easy to find $\|\hat{g}\|_2^2 \leq \|g^*\|_2^2 \leq \varepsilon^2$. Moreover, using Jensen's inequality on convex function $h(f_i + \cdot)$, we obtain
%\begin{displaymath}
%   h(f_i + \hat{g}(v)) \leq \frac{1}{\Delta v_i} \int_{S_{i-1}}^{S_i} h(f(v) + g^*(v)) \,\mathrm{d}v
%\end{displaymath}
\begin{equation} \label{Jensen}
\int_{S_{i-1}}^{S_i} h(f(v) + \hat{g}(v)) \,\mathrm{d}v = \Delta v_i h(f_i + \hat{g}(v)) \leq
\int_{S_{i-1}}^{S_i} h(f_i + g^*(v)) \,\mathrm{d}v.
\end{equation}
Note that $g^*(\cdot)$ is the optimal solution, implying that the equality must hold for \eqref{Jensen}.
However, since $h(f_i +\cdot)$ is strictly convex, %and $f(v)$ is a constant on $(S_{i-1}, S_i]$, 
the equality holds only when $g^*(v)$ is a constant on $(S_{i-1}, S_i]$, which contradicts our assumption. This completes the proof of the corollary.
\end{proof}

Another important corollary of \cref{lm: opt} is to pick $\beta_1 = O( \| \boldf^{n+1} - \boldf(t_{n+1}) \|_2)$ and construct $\boldf^{n+1,1}$ following \cref{eq:thm2f1}, such that the entropy of $\boldf^{n+1,1}$ is less than the entropy of $\boldf(t_{n+1})$ in the range of $i \leq I_1$.

\begin{corollary} \label{col: H1}
	Let $\varepsilon := \| \boldf^{n+1} - \boldf(t_{n+1}) \|_2$. Suppose $\boldf^{n+1}$ satisfies the condition of \cref{lm: opt} and $\varepsilon < \frac{\sqrt{C_1 V}C_f}{2}$. Let $\beta_1 = \frac{2 \varepsilon}{\sqrt{C_1 V} C_f}$ and
	\begin{equation}
		\boldf^{n+1,1} = \boldf^{n+1} + \beta_1 (\boldone - \boldf^{n+1}).
	\end{equation}
	Then $\boldf^{n+1,1}$ satisfies
	\begin{equation} \label{eq:colopt}
		\sum_{i=1}^{I_1} h(f^{n+1,1}_i) \Delta v_i \leq \sum_{i=1}^{I_1} h(f_i(t_{n+1})) \Delta v_i.
% \leq \min_{\Delta v \sum_{i=1}^{I_1} g_i^2 \leq \varepsilon^2} \Delta v \sum_{i=1}^{I_1} h(f^{n+1}_i + g_i)
	\end{equation}
\end{corollary}

\begin{proof}
    Let $g_1^*, \ldots, g_{I_1}^*$ be the solution of the optimization problem \cref{eq:opt}. Since
    \begin{displaymath}
        \sum_{i=1}^{I_1} (f_i^{n+1} - f_i(t_{n+1}))^2 \Delta v_i \leq \| \boldf^{n+1} - \boldf(t_{n+1}) \|_2^2 = \varepsilon^2,
    \end{displaymath}
    it holds that
	\begin{displaymath}
	    \sum_{i=1}^{I_1} h(f^{n+1}_i + g_i^*) \Delta v_i \leq \sum_{i=1}^{I_1} h(f_i(t_{n+1})) \Delta v_i.
	\end{displaymath}
	To prove \cref{eq:colopt}, it suffices to show 
	\begin{equation} \label{eq:bridgeminpro}
	    \sum_{i=1}^{I_1} h(f^{n+1,1}_i) \Delta v_i \leq \sum_{i=1}^{I_1} h(f^{n+1}_i + g_i^*) \Delta v_i.
	\end{equation}
	By the conclusion of \cref{lm: opt},
	\begin{displaymath}
		f_i^{n+1,1} = f_i^{n+1} + \beta_1 (1 - f_i^{n+1}) \geq  f_i^{n+1} + \frac{\beta_1}{2} = f_i^{n+1} + \frac{\varepsilon}{\sqrt{C_1 V} C_f} \geq f_i^{n+1} + g_i^*,
	\end{displaymath}
	for all $1 \leq i \leq I_1$. Noticing that $\beta_1 < 1$ by the constraint $\varepsilon < \frac{\sqrt{C_1 V} C_f}{2}$, we obtain $f_i^{n+1,1} < 1$. Hence, the monotonicity of $h(\cdot)$ yields
	\begin{displaymath}
		h(f_i^{n+1,1}) \leq h(f_i^{n+1} + g_i^*), \qquad \forall i = 1,\ldots, I_1.
	\end{displaymath}
	Multiplying $\Delta v_i$ and summing up the above inequalities for $i$ yields \cref{eq:bridgeminpro}. 
\end{proof}

By \cref{col: H1}, we have performed our first step that reduce the entropy of the smallest part of $\boldf^{n+1}$ (from $f_1^{n+1}$ to $f_{I_1}^{n+1}$) below the entropy of the exact solution in the same section. If the smallest component beyond this section $f_{I_1+1}^{n+1}$ already has the magnitude $O(1)$, for instance, $f_{I_1+1}^{n+1} \geq \frac{1}{2}$, then the remaining part can be processed using the same technique as in \cref{thmlog} and \cref{thm2}. Therefore, below we will only focus on the case where $f_{I_1+1}^{n+1} < 1/2$, and this inspires us to further decompose the remaining components into two parts by introducing $I_2$ such that
\begin{equation} \label{eq:I2}
	f_{I_2}^{n+1} \leq \frac{1}{2}, \qquad f_{I_2+1}^{n+1} > \frac{1}{2}.
\end{equation}
Then we will have $\eta(\boldf) - V = H_1(\boldf) + H_2(\boldf) + H_3(\boldf)$ for any $\boldf \in \mathbb{R}_+^N$, where
\begin{equation} \label{eq:decomH}
	H_1(\boldf) = \sum_{i=1}^{I_1} h(f_i) \Delta v_i,\quad H_2(\boldf) = \sum_{i=I_1+1}^{I_2} h(f_i) \Delta v_i, \quad H_3(\boldf) = \sum_{i=I_2+1}^{N} h(f_i) \Delta v_i.
\end{equation}
Note that this decomposition also includes the case $f_{I_1+1}^{n+1} \geq \frac{1}{2}$, for which we can choose $I_2 = I_1$, so that $H_2(\boldf^{n+1}) = 0$.

\cref{lemma:h2h1} will show some properties of above decomposition. Before that, a quotient $F(x,y,C)$, which will be used in the proof of \cref{lemma:h2h1}, is introduced as
\begin{equation} \label{eq:quotientf}
    F(x,y,C) = \frac{h(x + y) - h(x + C y)}{h(x) - h(x + y)}, 
%\qquad G(x,C) = F(x, \frac{1}{2C}, C),
\end{equation}
where $0 \leq x \leq 1/2$, $C > 1$ and $0 \leq y \leq 1/(2C)$. It is easy to find $F(x,y,C) \geq 0$ in its domain of definition. Furthermore, the following lemma gives the positive lower bound of $F(x,y,C)$ for fixed $C$, where the proof utilizes the (partial) derivatives of $F(x,y,C)$ and its detail is left in \cref{appendix:flowbound}.
%supplementary material \ref{sup:flowbound}.

\begin{lemma} \label{lemma:twoquotients}
    For any $C_1 \in (0,1]$, there exists $C_2>1$ depending on $C_1$, such that $F(x,y,C_2)$ given in \cref{eq:quotientf} satisfies
    \begin{displaymath}
        F(x,y,C_2) \geq \frac{1}{C_1}, \qquad \forall 0 \leq x \leq 1/2,\  0 \leq y \leq \frac{1}{2C_2}.
    \end{displaymath}
    % \begin{enumerate}
    %     \item $\frac{\partial F(x, y, C)}{\partial y} \leq 0$;
    %     \item $G(x,C) \geq \min (G(0, C), G(1/2, C))$ for fixed $C$.
    % \end{enumerate}
\end{lemma}

% \begin{proof}
%     The proof utilizes the (partial) derivatives of $F(x,y,C)$ and $G(x,C)$. Details are left in \cref{appendix:fbeta} and \cref{appendix:ffi} for two statements, respectively.
% \end{proof}

\begin{lemma} \label{lemma:h2h1}
	Under the condition of \cref{col: H1} and the decomposition of \cref{eq:decomH}, the following properties are satisfied:
	\begin{enumerate}
		\item $H_2(\boldf^{n+1,1}) - H_2(\boldf(t_{n+1})) \leq \frac{1}{C_1} (H_1(\boldf^{n+1}) - H_1(\boldf^{n+1,1}))$;
		\item There exists a constant $M_1>1$ depending on $C_1$ such that when $\varepsilon \leq \frac{\sqrt{C_1 V} C_f}{2M_1}$,
		%(recall $\beta_1 = \frac{2 \varepsilon}{\sqrt{C_1 V} C_f}$),
		the vector
		\begin{equation} \label{eq:beta2}
			\boldf^{n+1,2} = \boldf^{n+1,1} + M_1 \beta_1 ( \boldone - \boldf^{n+1,1})
		\end{equation}
		satisfies $H_1(\boldf^{n+1,1}) - H_1(\boldf^{n+1,2}) \geq \frac{1}{C_1}  (H_1(\boldf^{n+1}) - H_1(\boldf^{n+1,1}))$.
	\end{enumerate}
\end{lemma}

\begin{proof}
	To prove the first statement, we use the convexity of $H_2(\cdot)$ to obtain
	\begin{equation} \label{eq:hmono}
		H_2(\boldf^{n+1,1}) = H_2(\boldf^{n+1}  + \beta_1 (\boldone - \boldf^{n+1})) \leq \max(H_2(\boldf^{n+1}), H_2(\boldone)) = H_2(\boldf^{n+1}).
	\end{equation}
	Therefore,
	\begin{equation}
	\label{eq:H2}
		\begin{aligned}
			& H_2(\boldf^{n+1,1}) - H_2(\boldf(t_{n+1})) \\
			={} & H_2(\boldf^{n+1,1}) - H_2(\boldf^{n+1}) + H_2(\boldf^{n+1}) - H_2(\boldf(t_{n+1})) \\
			\leq{} & H_2(\boldf^{n+1}) - H_2(\boldf(t_{n+1})) 
			\leq H_2(\boldf^{n+1}) - H_2(\boldf^{n+1} + \boldg^{**}),
		\end{aligned}
	\end{equation}
	where $\boldg^{**}=(g^{**}_1, \ldots, g^{**}_N)^T \in \mathbb{R}^N$ is the solution of following minimization problem:
	\begin{displaymath}
		\argmin_{\| \boldg \|_2 \leq \varepsilon} H_2(\boldf^{n+1} + \boldg).
	\end{displaymath}
	The existence of $\boldg^{**}$ is because $H_2(\boldf^{n+1} + \boldg)$ is a continuous function (w.r.t. $\boldg$) defined on a closed set and the constrain $\| \boldg \|_2 \leq \varepsilon$ also gives a closed set for $\boldg$. 
	The solution $\boldg^{**}$ satisfies that $g^{**}_i \geq 0$ for all $I_1 < i \leq I_2$, since replacing any negative component of $\boldg$ by zero will lead to a smaller value for the objective function.%If not, suppose there is $g^{**}_k<0$ for one $k$, where $I_1 < k \leq I_2$. Since $f^{n+1}_k \leq \frac{1}{2}$, it holds that $g_k^{**} \geq -\frac{1}{2}$. Therefore, $1 \geq f^{n+1}_k - g_k^{**}>f^{n+1}_k + g_k^{**}$. Then the strict monotonicity of $h(\cdot)$ in $[0,1]$ gives $h(f^{n+1}_k - g_k^{**}) < h(f^{n+1}_k + g_k^{**}) $. We could construct $\bar{\boldg}$ such that \bo{this paragraph to be one line, using KKT condition?}
	%\begin{displaymath}
	%	\bar{g}_i = \left\{ 
	%	\begin{array}{ll}
	%		g^{**}_i, 	&   i \neq k,\\
	%		-g^{**}_i,  &   i = k.
	%	\end{array}
	%	\right.
	%\end{displaymath}
	%We could find $\| \bar{\boldg} \|_2 = \| \boldg^{**} \|_2 \leq \varepsilon$ and $H_2[\boldf^{n+1} + \bar{\boldg}] < H_2[\boldf^{n+1} + \boldg^{**}] $, which contradicts with the optimality of $\boldg^{**}$.
	
	For any $i = I_1+1, \ldots, I_2$ and $j = 1,\ldots,I_1$, the convexity of $h(\cdot)$ implies
	\begin{equation} \label{eq:hitohj}
	h(f_i^{n+1}) - h(f_i^{n+1} + g^{**}_i) \leq h(f^{n+1}_j) - h(f^{n+1}_j + g^{**}_i).
	\end{equation}
	To extend the above inequality to functions defined on $\mathbb{R}_+$ with support in $[0, V]$, which is convenient for our proof in the following step, we would like to follow the notation in \cref{eq: partialSum} and
%	introduce the (partial) sum of first $i$ parameters $\Delta v_i$ as
%	\begin{displaymath}
%	   S_0 = 0, \qquad S_i = \sum_{j=1}^i \Delta v_j, \qquad i = 1, \dots, N.
%	\end{displaymath}
%	Then we could 
    represent $\boldf^{n+1}$ and $\boldg^{**}$ by piesewise constant functions $f^{n+1}(v)$ and $g^{**}(v)$ respectively as
	\begin{displaymath}
	   f^{n+1}(v) = f^{n+1}_i, \qquad g^{**}(v) =  g^{**}_i, \qquad v \in (S_{i-1}, S_i],\  i = 1, \dots, I_2,
	\end{displaymath}
	and both $f^{n+1}(v)$ and $g^{**}(v)$ equal zero if $v > S_{I_2}$. %The values for the above two functions at $v=S_i$ is not important, which could be defined as either the left limit or the right one.
	Using the functions $f^{n+1}(v)$ and $g^{**}(v)$, the inequality \cref{eq:hitohj} is equivalent to: for any $w \in (S_{I_1}, S_{I_2})$ and $v \in (S_0, S_{I_1})$,
	\begin{displaymath}
	   h(f^{n+1}(w)) - h(f^{n+1}(w) + g^{**}(w)) \leq h(f^{n+1}(v)) - h(f^{n+1}(v) + g^{**}(w)).
	\end{displaymath}
	Since $g^{**}(w)=0$ for $w \geq S_{I_2}$, the above inequality actually holds for any $w \in (S_{I_1}, + \infty)$. Therefore, we choose $w = v + k S_{I_1}$ with $k \geq 1$ to obtain
	\begin{equation}
	\label{eq:H2diff}
		\begin{aligned}
			& H_2(\boldf^{n+1}) - H_2(\boldf^{n+1} + \boldg^{**}) \\
			={} & \sum_{i=I_1+1}^{I_2} \left( h(f^{n+1}_i) - h(f^{n+1}_i + g^{**}_i) \right) \Delta v_i \\
			={} & \int_{S_{I_1}}^{S_{I_2}} \left( h(f^{n+1}(v)) - h(f^{n+1}(v) + g^{**}(v)) \right) \mathrm{d}v \\
			= {} & \sum_{k=1}^{\left\lceil \frac{S_{I_2}-S_{I_1}}{S_{I_1}} \right\rceil} \int_{0}^{S_{I_1}} \left( h(f^{n+1}(v+k S_{I_1})) - h(f^{n+1}(v+k S_{I_1}) + g^{**}(v + k S_{I_1})) \right) \mathrm{d}v \\
			\leq{} & \sum_{k=1}^{\left\lceil \frac{S_{I_2}-S_{I_1}}{S_{I_1}} \right\rceil} \int_{0}^{S_{I_1}} \left( h(f^{n+1}(v)) - h(f^{n+1}(v) + g^{**}(v + k S_{I_1})) \right) \mathrm{d}v.
		\end{aligned}
	\end{equation}
	Since $\| g^{**}\|_2^2 \leq \| \boldg^{**} \|_2^2 \leq \varepsilon^2$, for any $1 \leq k \leq \lceil \frac{S_{I_2}-S_{I_1}}{S_{I_1}} \rceil$, we have
	\begin{equation}
	\label{eq:hineq}
		\begin{aligned}
        	& \int_{0}^{S_{I_1}} h(f^{n+1}(v) + g^{**}(v + k S_{I_1})) \mathrm{d}v \\
        	\geq {} & \int_{0}^{S_{I_1}} h(f^{n+1}(v) + g^{*}(v)) \mathrm{d}v  \\
        	= {} & \sum_{j=1}^{I_1} h(f^{n+1}_j + g_j^*) \Delta v_j \geq
        	\sum_{j=1}^{I_1} h(f^{n+1,1}_j) \Delta v_j,
    	\end{aligned}
	\end{equation}
	where $g^*(v)$ and $g_i^*$ stand for the solutions of the optimization problem \cref{eq: contOptproblem} and \cref{eq:opt}, respectively; the equality is the conclusion of \cref{lm: optcont}, and the last ``$\geq$'' comes from the inequality \cref{eq:bridgeminpro}. Inserting \cref{eq:hineq} into \cref{eq:H2diff} yields
	\begin{equation}
	\label{eq:H2est}
	    \begin{aligned}
	        H_2(\boldf^{n+1}) - H_2(\boldf^{n+1} + \boldg^{**}) &\leq \sum_{k=1}^{\lceil \frac{S_{I_2} - S_{I_1}}{S_{I_1}} \rceil} \sum_{j=1}^{I_1} \left( h(f^{n+1}_j) - h(f^{n+1,1}_j) \right) \Delta v_j \\
	        & \leq \frac{V}{S_{I_1}} (H_1(\boldf^{n+1}) - H_1(\boldf^{n+1,1})).
	    \end{aligned}
	\end{equation}
	Since the definition of $I_1$ implies $S_{I_1} \geq C_1 V$, concatenating \cref{eq:H2} and \eqref{eq:H2est} proves the first statement.
	
	The second statement will be proved componentwisely. We set $M_1 = 2C_2$, where $C_2$ is determined by \cref{lemma:twoquotients} with $C_1$ being chosen as the constant $C_1$ appearing in the first statement.
	%According to \cref{lemma:twoquotients}, we can find $C_2>1$ depending on $C_1$, such that $F(x, y, C_2) \geq \frac{1}{C_1}$ (recall $F(x, y, C_2)$ comes from \cref{eq:quotientf}) for all $0 \leq x \leq \frac{1}{2}$ and $0 \leq y \leq \frac{1}{2C_2}$. With $C_2$, the constant $M$ in the second statement could be taken as $M = 2 C_2$.
	Then, for any $1 \leq i \leq I_1$, it holds that
	\begin{displaymath}
		f_i^{n+1,1} = f_i^{n+1} + \beta_1 (1 - f^{n+1}_i) \leq f_i^{n+1} + \beta_1.
	\end{displaymath}
	Moreover, when $\varepsilon \leq \frac{\sqrt{C_1 V} C_f}{2M_1}$, it could be found that $\beta_1 \leq 1 / M_1$ and
	\begin{displaymath}
	\begin{aligned}
		f_i^{n+1,2} &= f_i^{n+1} + (\beta_1 + M_1 \beta_1 - M_1 \beta_1^2) (1 - f^{n+1}_i) \\
		& \geq f_i^{n+1} + M_1 \beta_1(1 - f^{n+1}_i) \geq f_i^{n+1} + C_2 \beta_1,
	\end{aligned}
	\end{displaymath}
	where we have used $f_i^{n+1} \leq \frac{1}{2}$ and $M_1 = 2 C_2$. Therefore, the monotonicity of $h(\cdot)$ in the interval of $[0,1]$ implies
	\begin{displaymath}
		\frac{h(f_i^{n+1,1}) - h(f_i^{n+1,2})}{h(f_i^{n+1}) - h(f_i^{n+1,1})} \geq \frac{h(f_i^{n+1} + \beta_1) - h(f_i^{n+1} + C_2 \beta_1)}{h(f_i^{n+1}) - h(f_i^{n+1}+ \beta_1)} = F(f_i^{n+1}, \beta_1, C_2) \geq \frac{1}{C_1},
	\end{displaymath}
	where the function $F(\cdot,\cdot,\cdot)$ is defined in \cref{eq:quotientf} and the last inequality is due to \cref{lemma:twoquotients}.
	%where we plug $x=f_i^{n+1}$ and $y=\beta_1$ in $F(x, y, C_2)$ and utilize $F(x, y, C_2) \geq \frac{1}{C_1}$ for all $0 \leq x \leq \frac{1}{2}$ and $0 \leq y \leq \frac{1}{2C_2}$.
	By noticing $h(f_i^{n+1}) - h(f_i^{n+1,1}) \geq 0$, the second statement can then be easily derived.
\end{proof}

With the preparation of \cref{lemma:h2h1}, we can start to prove \cref{thm3}.

\begin{proof}[Proof of \cref{thm3}]
%	If $f_1^{n+1} = 0$, we could first make a perturbation or regularity on $\boldf^{n+1}$ as $\boldf^{n+1} + \varepsilon (\boldone - \boldf^{n+1})$, after which all components $f_i^{n+1} + \varepsilon (1 - f_i^{n+1}) \geq \varepsilon$ and the $L^2$ norm of perturbation $\|\varepsilon (\boldone - \boldf^{n+1}) \|_2 \leq \sqrt{V} \| \boldf^{n+1} \|_{\infty} \| \boldf^{n+1} - \boldf(t_{n+1}) \|_{2}$. Therefore, without loss of generality, we could just prove the case for $f_1^{n+1} > 0$ and
%	\begin{equation} \label{eq:logratio}
%		\frac{\log \left( f_i^{n+1} \right)}{ \log \left( f_1^{n+1} \right) } \geq C_f.
%	\end{equation}
	
	If $f_{I_1}^{n+1} \geq \frac{1}{2}$, \cref{eq:maxlogratio} implies $\log(f_1^{n+1}) \geq -\frac{1}{C_f} \log(2)$, which means $f_1^{n+1} \geq 2^{-1/C_f}$. Then, from \cref{thm1}, we get $\| \beta  ( \boldone - \boldf^{n+1})\|_2 \leq M \| \boldf^{n+1} - \boldf(t_{n+1}) \|_{2}$ where $M>0$ depends on $2^{-1/C_f}$, $\| \boldf^{n+1} \|_{\infty}$ and $\| \boldf(t_{n+1})\|_{\infty}$. This completes the proof.
	
	Otherwise, if $f_{I_1}^{n+1} < \frac{1}{2}$, we would like to introduce $I_2$ and decompose $\eta(\boldf)$ following \cref{eq:I2} and \cref{eq:decomH}. After that, we construct $\boldf^{n+1,1}$ and $\boldf^{n+1,2}$ from \cref{eq:thm2f1} with $\beta_1 = \| \boldf(t_{n+1}) - \boldf^{n+1} \|_{2} / (\sqrt{C_1 V} C_f)$ and \cref{eq:beta2} with $\beta_2 = M_1 \beta_1$, respectively, where the $M_1$ is the constant in \cref{lemma:h2h1}.
	Then we set $\delta = \sqrt{C_1 V}C_f/2$, and if $\| \boldf^{n+1} - \boldf(t_{n+1}) \| < \delta$, it holds that
	\begin{displaymath}
 	    \begin{aligned}
 			& H_1(\boldf^{n+1,2}) + H_2(\boldf^{n+1,2}) - H_1(\boldf(t_{n+1})) - H_2(\boldf(t_{n+1})) \\
 			={} & \left( H_1(\boldf^{n+1,2}) - H_1(\boldf^{n+1,1}) \right) + \left( H_1(\boldf^{n+1,1}) -  H_1(\boldf(t_{n+1})) \right) \\
 			& + \left( H_2(\boldf^{n+1,2}) - H_2(\boldf^{n+1,1}) \right) + \left( H_2(\boldf^{n+1,1}) -  H_2(\boldf(t_{n+1})) \right) \\
 			\leq{} & \left( H_1(\boldf^{n+1,2}) - H_1(\boldf^{n+1,1}) \right) + 0 & \text{(\cref{col: H1})} \\
 		    & + \left( - \frac{H_1(\boldf^{n+1}) - H_1(\boldf^{n+1,1})}{C_1} \right) + \left(  \frac{H_1(\boldf^{n+1}) - H_1(\boldf^{n+1,1})}{C_1} \right) & \text{(\cref{lemma:h2h1})} \\
 			={} &  H_1(\boldf^{n+1,2}) - H_1(\boldf^{n+1,1}) \leq 0,
    	\end{aligned}
 	\end{displaymath}
    where the last inequality is similar to \cref{eq:hmono} which utilizes the convexity of $H_1(\cdot)$.
% 	It holds that
	
% 	%$\beta_2 = 2 \left( \frac{2(1+C_1)}{C_1(1+\log(2))} \right)^2 \beta_1$
% 	%\begin{displaymath}
% 	%	\begin{split}
% 	\begin{align*}
% 			& H_1[\boldf^{n+1,2}] + H_2[\boldf^{n+1,2}] - H_1[\boldf(t_{n+1})] - H_2[\boldf(t_{n+1})] \\
% 			={} & H_1[\boldf^{n+1,2}] - H_1[\boldf^{n+1,1}] + H_1[\boldf^{n+1,1}] -  H_1[\boldf(t_{n+1})] \\
% 			& + H_2[\boldf^{n+1,2}] - H_2[\boldf^{n+1,1}] + H_2[\boldf^{n+1,1}] -  H_2[\boldf(t_{n+1})].
% 	\end{align*}
% 	%	\end{split}
% 	%\end{displaymath}
% 	Firstly, similar to \cref{eq:hmono}, it is easy to find $ H_2[\boldf^{n+1,2}] - H_2[\boldf^{n+1,1}] \leq 0$. Then, let $\delta = \sqrt{C_1 V}C_f/2$ and if $\| \boldf^{n+1} - \boldf(t_{n+1}) \| < \delta$, \cref{col: H1} illustrates $ H_1[\boldf^{n+1,1}] -  H_1[\boldf(t_{n+1})] \leq 0$. Lastly, \cref{lemma:h2h1} shows $H_2[\boldf^{n+1,1}] -  H_2[\boldf(t_{n+1})] \leq \frac{1}{C_1} (H_1[\boldf^{n+1}] - H_1[\boldf^{n+1,1}])$ and $H_1[\boldf^{n+1,2}] - H_1[\boldf^{n+1,1}] \leq - \frac{1}{C_1} (H_1[\boldf^{n+1}] - H_1[\boldf^{n+1,1}])$. As a result, we could find 
% 	\begin{displaymath}
% 		H_1[\boldf^{n+1,2}] + H_2[\boldf^{n+1,2}] - H_1[\boldf(t_{n+1})] - H_2[\boldf(t_{n+1})] \leq 0.
% 	\end{displaymath}
% 	which means 
    Therefore, by the decomposition in \cref{eq:decomH},
	\begin{equation}
	\label{eq:Hdiff}
		\begin{aligned}
			\eta(\boldf^{n+1,2}) - \eta(\boldf(t_{n+1})) & \leq H_3(\boldf^{n+1,2}) - 	H_3(\boldf(t_{n+1})) \\
			& = \sum_{f_i^{n+1}> \frac{1}{2} } \left( h(f_i^{n+1,2}) - h(f_i(t_{n+1}))\right) \Delta v_i.
		\end{aligned}
	\end{equation}
	From the construction of $\boldf^{n+1,2}$, we know $f_i^{n+1,2}$ is a convex combination of $1$ and $f_i^{n+1}$, so $f_i^{n+1}> \frac{1}{2}$ implies $f_i^{n+1,2}> \frac{1}{2}$. Therefore \cref{eq:Hdiff} can be further extended as
	\begin{equation}
	\label{eq:Hdiff1}
		\eta(\boldf^{n+1,2}) - \eta(\boldf(t_{n+1})) \leq \sum_{f_i^{n+1,2}> \frac{1}{2} } \left( h(f_i^{n+1,2}) - h(f_i(t_{n+1}))\right) \Delta v_i.
	\end{equation}
	
	The remaining part of the proof is similar to the proof of \cref{thm2}. If $\eta(\boldf^{n+1,2}) \leq \eta(\boldf(t_{n+1}))$, the proof is done. Otherwise, we have $\eta(\boldf^{n+1,2}) > \eta(\boldf(t_{n+1}))$, and we can continue to find $\hat{\boldf}^{n+1}$ and $\beta_3$ such that
	\begin{displaymath}
		\hat{\boldf}^{n+1} = \boldf^{n+1,2} + \beta_3 ( \boldone - \boldf^{n+1,2}),
	\end{displaymath}
	and $\eta(\hat{\boldf}^{n+1}) = \eta(\boldf(t_{n+1}))$. Due to the inequality \cref{eq:Hdiff1}, we can follow the proof of \cref{lm:hdiff} (case (ii)) and \cref{thm1} to show
	\begin{displaymath}
		\| \beta_3  ( \boldone - \boldf^{n+1,2})\|_2 \leq M_2 \| \boldf(t_{n+1}) - \boldf^{n+1,2} \|_2,
	\end{displaymath}
	where $M_2>0$ is a constant depending on $\| \boldf^{n+1} \|_{\infty}$ (because $\| \boldf^{n+1,2} \|_{\infty} \leq \| \boldf^{n+1} \|_{\infty}$) and $\| \boldf(t_{n+1})\|_{\infty}$. Therefore, $\eta(\hat{\boldf}^{n+1}) \leq \eta(\boldf(t_{n+1}))$, and
	%\begin{displaymath}
	%	\begin{split}
	\begin{align*}
			& \| \hat{\boldf}^{n+1} - \boldf^{n+1} \|_2 \\
			\leq{} &  \| \hat{\boldf}^{n+1} - \boldf^{n+1,2} \|_2 + \| \boldf^{n+1,2} - \boldf^{n+1} \|_2 \\
			\leq{} & M_2 \| \boldf(t_{n+1}) - \boldf^{n+1,2} \|_{2} + \| \boldf^{n+1,2} - \boldf^{n+1} \|_2 \\
			\leq{} & M_2 \| \boldf(t_{n+1}) - \boldf^{n+1} \|_{2} + (1 + M_2)\| \boldf^{n+1,2} - \boldf^{n+1} \|_2 \\
			={} & M_2 \| \boldf(t_{n+1}) - \boldf^{n+1} \|_{2} + (1 + M_2)(\beta_1 + \beta_2 - \beta_1 \beta_2)\| \boldone - \boldf^{n+1} \|_2 \\
			\leq{} & \left(M_2 + \frac{(1 + M_1) (1 + M_2) \| \boldf^{n+1} \|_{\infty}}{\sqrt{C_1} C_f} \right) \| \boldf(t_{n+1}) - \boldf^{n+1} \|_{2},
	\end{align*}
	%	\end{split}
	%\end{displaymath}
	where the last inequality utilizes $(\beta_1 + \beta_2 - \beta_1 \beta_2) \leq \beta_1 + \beta_2$ and $\| \boldone - \boldf^{n+1} \|_2 \leq \sqrt{V} \| \boldf^{n+1} \|_{\infty}$. If we denote the constant in front of $\| \boldf(t_{n+1}) - \boldf^{n+1} \|_{2}$ as $M$, we have proved the constructed 
	%\begin{displaymath}
	%	\begin{split}
	\begin{align*}
			\hat{\boldf}^{n+1} & = \boldf^{n+1,2} + \beta_3 ( \boldone - \boldf^{n+1,2}) \\
			& = \boldf^{n+1} + (\beta_1 + \beta_2 + \beta_3 - \beta_1 \beta_2 - \beta_2 \beta_3 - \beta_1 \beta_3 + \beta_1 \beta_2 \beta_3 ) (\boldone - \boldf^{n+1}),
	\end{align*}
	%	\end{split}
	%\end{displaymath}
	such that $\eta(\hat{\boldf}^{n+1}) \leq \eta(\boldf(t_{n+1}))$ and $\| \hat{\boldf}^{n+1} - \boldf^{n+1} \|_2 \leq M \| \boldf(t_{n+1}) - \boldf^{n+1} \|_{2}$. Due to the monotonicity of $H(\boldf^{n+1} + \beta(\boldone- \boldf^{n+1}))$ w.r.t. $\beta$, if we construct $\beta$ from \cref{eq:enfix},
	\begin{displaymath}
		\| \beta  ( \boldone - \boldf^{n+1})\|_2 \leq \| \hat{\boldf}^{n+1} - \boldf^{n+1} \|_2  \leq M \| \boldf^{n+1} - \boldf(t_{n+1}) \|_{2}.
	\end{displaymath}
\end{proof}

\section{Numerical examples}
\label{sec:experiments}
We now present two numerical examples to show the effect of our entropy fix. In order to construct cases where the numerical scheme frequently violates the entropy inequality, we deliberately select highly oscillatory initial data. We would like to remark that such an entropy fix may only need to be applied occasionally in many applications.

\subsection{Linear Fokker-Planck equation}
In this example, we consider the one-dimensional linear Fokker-Planck equation (also known as the drift-diffusion equation):
\begin{equation}\label{eq:fokker}
f_t=f_{xx}+(V'(x)f)_x, \qquad t>0, x \in (0,1),
\end{equation}
with periodic boundary condition $f(t,0) = f(t,1)$ and potential function
\begin{displaymath}
   V(x) = \frac{1}{2 \pi} \cos \left( 20 \pi x \right).
\end{displaymath}
Let $M(x)=\exp(-V(x))$, then \cref{eq:fokker} can be written equivalently as
\begin{equation} \label{FP1}
    f_t=\left(M \left(\frac{f}{M}\right)_x\right)_x.
\end{equation}
If we further define $g(t,x)=f(t,x) / M(x)$, then (\ref{FP1}) becomes
\begin{equation}\label{eq:fokkervarg}
    g_t=\frac{1}{M}\left(M g_x\right)_x, \qquad t>0, x \in (0,1).
\end{equation}
%The above transformation reformulates the original Fokker-Planck equation as a variable coefficient diffusion hence can be naturally discretized by central finite difference.

We will focus on the discretization of \cref{eq:fokkervarg}. Initial condition is taken as
\begin{displaymath}
g(0,x) = 1.2 + \sum_{j=1}^{20} \frac{j}{210} \sin \left( 2 j \pi x \right).
\end{displaymath}
Note that $\sum_{j=1}^{20} j=210$, so $0.2 \leq g(0,x) \leq 2.2$. We partition $[0,1]$ into $N = 64$ grids uniformly with mesh size $\Delta x = 1/N$ and take central difference for spatial discretization. Denote $g_j = g(t, j \Delta x)$, $M_j = M(j \Delta x)$ and $M_{j+1/2} = M( (j+1/2)\Delta x)$ for $j= 0, \dots, N-1$, \cref{eq:fokkervarg} can be approximated by
\begin{equation} \label{eq:fokkergspatial}
   \frac{\mathrm{d} g_j}{\mathrm{d} t} = \frac{1}{M_j} \frac{M_{j+1/2}(g_{j+1}-g_j)-M_{j-1/2}(g_j-g_{j-1})}{(\Delta x )^2}.
\end{equation}
The exact solution of \cref{eq:fokkergspatial} can be calculated by evaluating the eigenvalues and eigenvectors of the right-hand side of \cref{eq:fokkergspatial}.

The semi-discrete scheme \cref{eq:fokkergspatial} (time is kept continuous) satisfies the conservation of mass and the monotonicity of entropy with weight $M_j$. In fact, it is easy to verify $\sum_{j=0}^{N-1} M_j g_j$ remains as constant. For the entropy, we have
\begin{equation}
\begin{split}
  \frac{\mathrm{d} \left( \sum_{j=0}^{N-1} M_j g_j\log g_j\right)}{\mathrm{d} t}&=\frac{1}{(\Delta x)^2}\sum_{j=0}^{N-1}\left(M_{j+1/2}(g_{j+1}-g_j)-M_{j-1/2}(g_j-g_{j-1})\right)\log g_j\\
  &=-\frac{1}{(\Delta x)^2}\sum_{j=0}^{N-1} M_{j-1/2}(g_j-g_{j-1})(\log g_j - \log g_{j-1}) \leq 0.
  \end{split}
\end{equation}

We now discretize \cref{eq:fokkergspatial} by the implicit midpoint (i.e., Crank–Nicolson) method. This time discretization still conserves the mass. However, there is no guarantee that the entropy will decay monotonically in time (in fact, it does not). In \cref{fig:fokker}, we report the time evolution of the entropy with and without the entropy fix. Two different time steps $\Delta t = 1/512$ and $\Delta t = 1/1024$ are considered. In both cases, it is clear that the entropy decreases monotonically with the help of the entropy fix. Meanwhile, the $L^2$ error of the solution remains almost the same with and without the entropy fix. It is interesting to note that when $\Delta t = 1/512$, the entropy fix is only needed at the first few time steps. On the other hand, when $\Delta t = 1/1024$, the entropy fix is required only after $t=0.02$.

%Consider the final time $T = 5 / 64$ with uniform time step $\Delta t$, denote $g_j^n = g_j(t=n\Delta t)$, the scheme reads
%\begin{equation} \label{scheme:im}
%\begin{aligned}
%\frac{g^{n+1}_j - g^n_{j}}{\Delta t} & = \frac{1}{M_j} \frac{M_{j+1/2}(g_{j+1}^{n+1}-g_j^{n+1})-M_{j-1/2}(g_j^{n+1}-g_{j-1}^{n+1})}{2 (\Delta x )^2} \\
%& + \frac{1}{M_j} \frac{M_{j+1/2}(g_{j+1}^{n}-g_j^{n})-M_{j-1/2}(g_j^{n}-g_{j-1}^{n})}{2 (\Delta x )^2}.
%\end{aligned}
%\end{equation}

\begin{figure}[htbp]
	\centering
\subfigure[Entropy vs $t$. $\Delta t = 1 / 512$.]
	{\includegraphics[width=.45\textwidth]{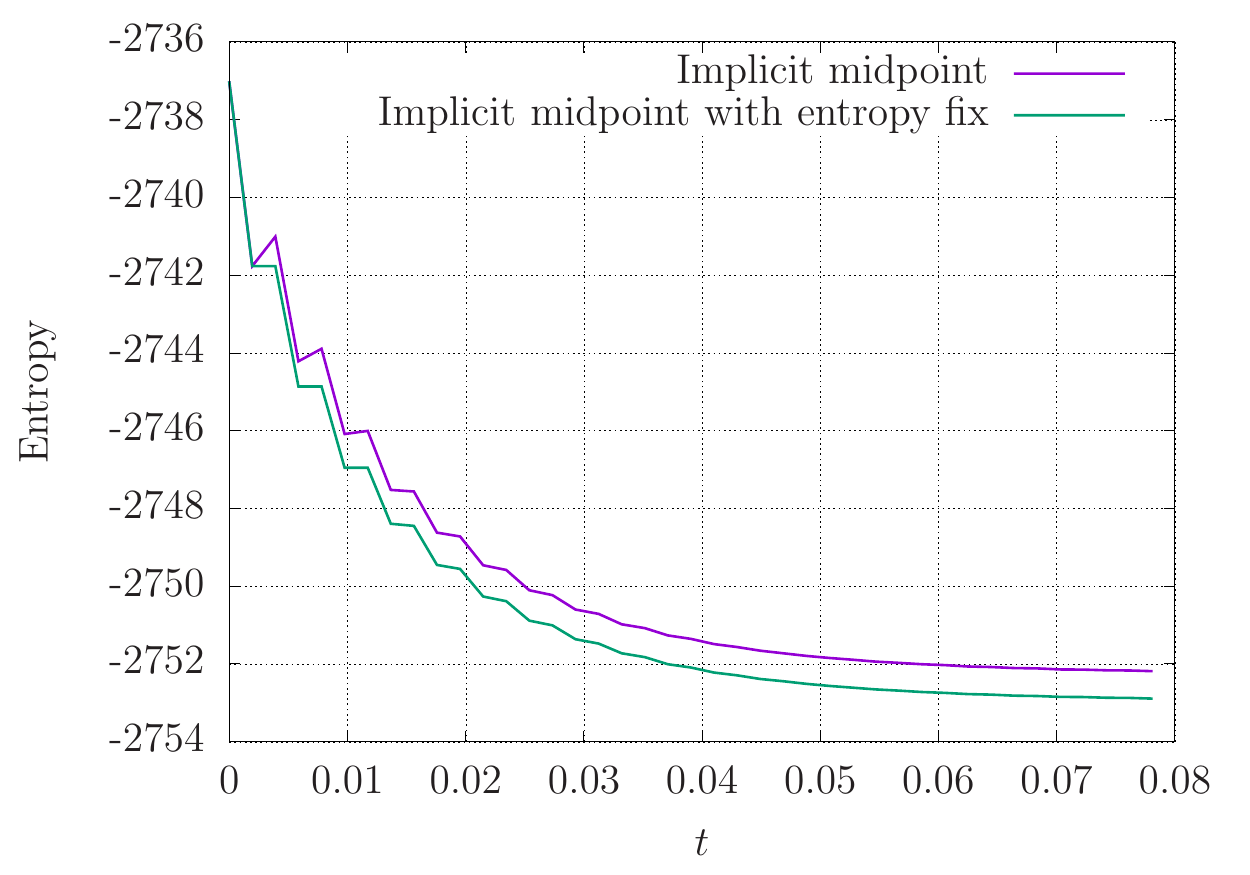}}
\subfigure[$L^2$ relative error vs $t$. $\Delta t = 1 / 512$.]
	{\includegraphics[width=.45\textwidth]{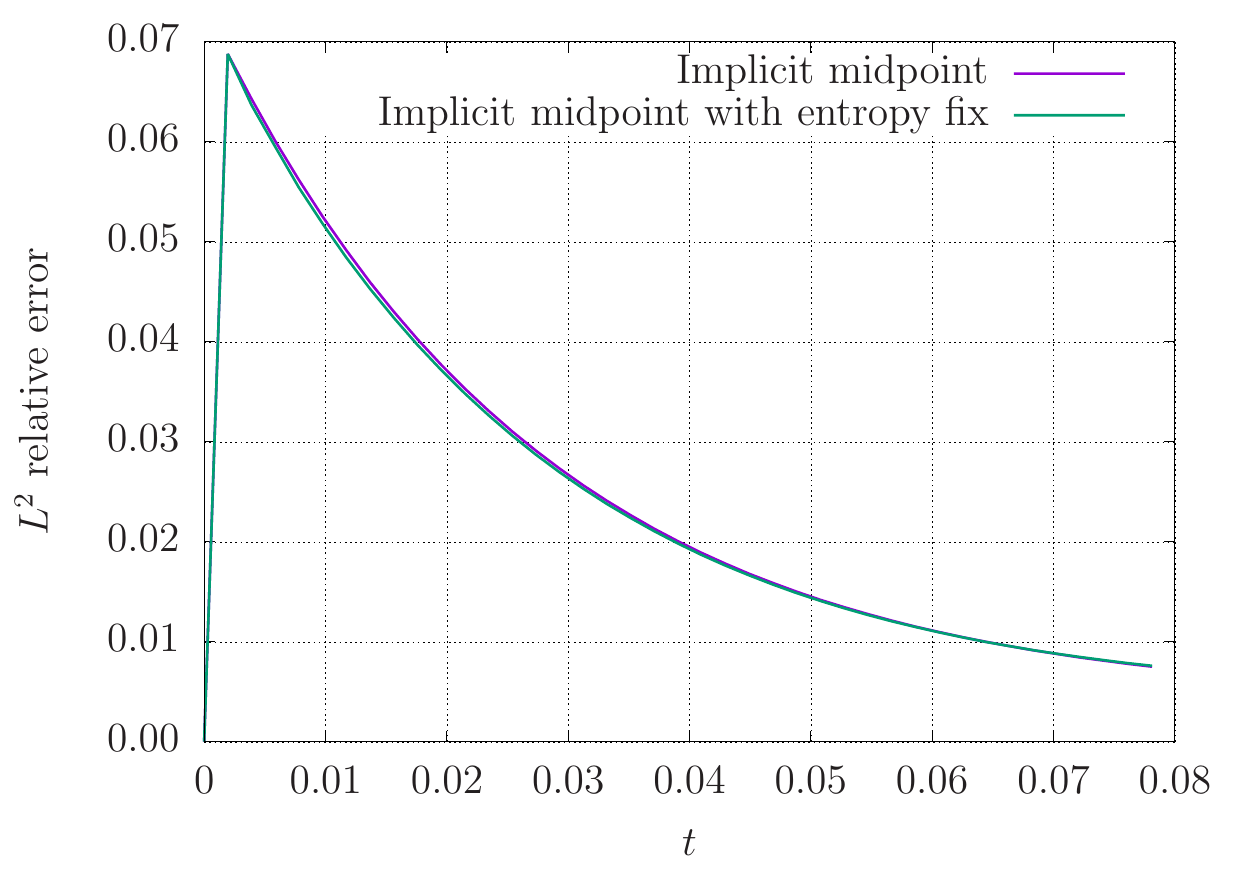}}
\subfigure[Entropy vs $t$. $\Delta t = 1 / 1024$.]
	{\includegraphics[width=.45\textwidth]{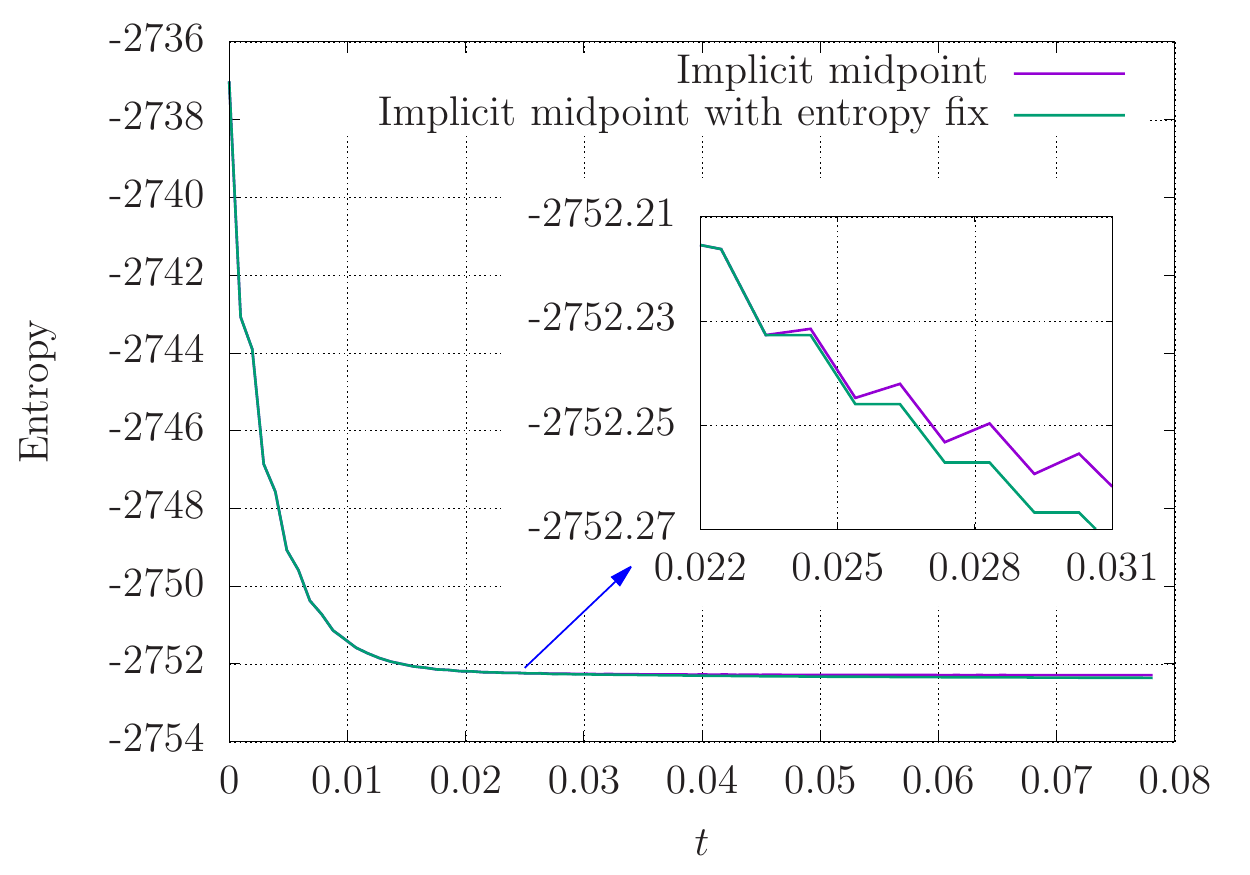}}
\subfigure[$L^2$ relative error vs $t$. $\Delta t = 1 / 1024$.]
	{\includegraphics[width=.45\textwidth]{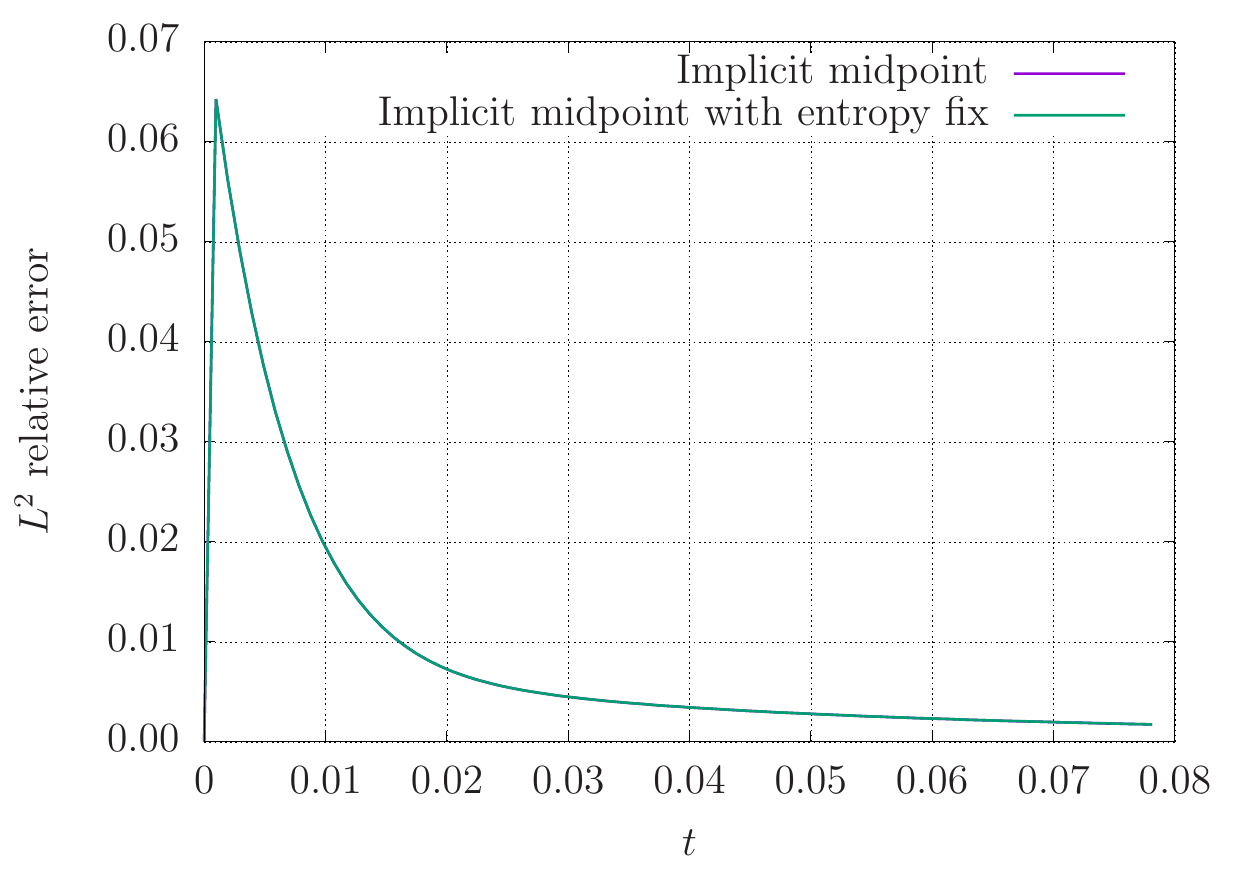}}
	
    \caption{Example of the linear Fokker-Planck equation. Time evolution of the entropy $H(\boldg)  = \sum_{j=0}^{N-1} (g_j \log g_j - g_j)M_j \Delta x $ and the $L^2$ relative error $\| \boldg - \boldg_{\rm exact} \|_2 / \| \boldg_{\rm exact} \|_2 = ( \sum_{j=0}^{N-1} ( g_j - g_{{\rm exact},j} )^2 M_j \Delta x )^{1/2} / ( \sum_{j=0}^{N-1} ( g_{{\rm exact},j} )^2 M_j \Delta x )^{1/2}$, where $\Delta x=1/64$, $\Delta t=1/512$ in the top two figures and $\Delta t=1/1024$ in the bottom two figures.
    \label{fig:fokker}}
\end{figure}

\subsection{Nonlinear Boltzmann equation} \label{example:boltzmann}
In this example, we consider a nonlinear model introduced in \cite{Cai2018entropic}, which results from a Fourier method for the spatially homogeneous Boltzmann equation. The governing equation reads
\begin{equation} \label{eq:Boltzmann}
    \frac{\mathrm{d} f_{r}(t)}{\mathrm{d} t}=\sum_{p, q, s \in \mathcal{X}} A_{p q}^{r s}\left(f_{p}(t) f_{q}(t)-f_{r}(t) f_{s}(t)\right), \quad r \in \mathcal{X},
\end{equation}
where $f_r$ represents the approximation  of the distribution function on a uniform 3D lattice index set $\mathcal{X} = \{(r_1, r_2, r_3) \mid r_i = 0,\dots,M-1 \text{ for } i = 1,2,3\}$.
In \cite{Cai2018entropic}, the coefficients $A_{p q}^{r s}$ are determined in such a way that the semi-discrete scheme (\ref{eq:Boltzmann}) decays the entropy. However, this property may not hold when the time is discretized.

%where the right-hand side approximates the Boltzmann collision operator, and
%the coefficients $A_{p q}^{r s}$ are nonnegative constants satisfying
%\begin{equation*}
%    A_{p q}^{r s}=A_{q p}^{r s}=A_{r s}^{p q}.
%\end{equation*}
%The index set $\mathcal{X}$ is given by
%\begin{displaymath}
%\mathcal{X} = \{(r_1, r_2, r_3) \mid r_i = 0,\dots,M-1 \text{ for } i = 1,2,3\},
%\end{displaymath}
%so that the unknowns $f_r$ represent the values of the distribution function on a uniform 3D lattice with $M^3$ points. In our experiments, we choose $M = 17$, and the values of $A_{pq}^{rs}$ are given in \cref{appendix:coeff}. 

In our experiment, we choose $M = 17$, and the values of $A_{pq}^{rs}$ are given in \cref{appendix:coeff}. The initial condition is taken as %\zc{Please check if the initial condition is correct.}
%For more details, please refer to \cite{Cai2018entropic}. We test the presented entropy fix method to this example with the initial condition given as
\begin{equation*}
    f_r(0) = 3.2 + \sum_{j=1}^{10} \frac{j}{55}\left[\sin \Big(j\pi \big(\frac{r_1}{M} - \frac{1}{2}\big)\Big)+\sin \Big(j\pi \big(\frac{r_2}{M} - \frac{1}{2}\big)\Big)+\sin \Big(j\pi \big(\frac{r_3}{M} - \frac{1}{2}\big)\Big)\right].
\end{equation*}
We solve \cref{eq:Boltzmann} by the forward Euler method with time step $\Delta t = 0.0007$. The results are displayed in \cref{fig:bkw3d}, from which we can see that the entropy fix method guarantees the monotonicity of the entropy. The numerical error is computed by comparison with the numerical solution computed with a smaller time step $\Delta t = 0.000175$, with and without the entropy fix. It can be seen that the two error curves almost coincide with each other, meaning that the entropy fix does not ruin the numerical accuracy.
\begin{figure}[htbp]
	\centering
	\subfigure[Entropy vs $t$.]
	{\includegraphics[width=.45\textwidth]{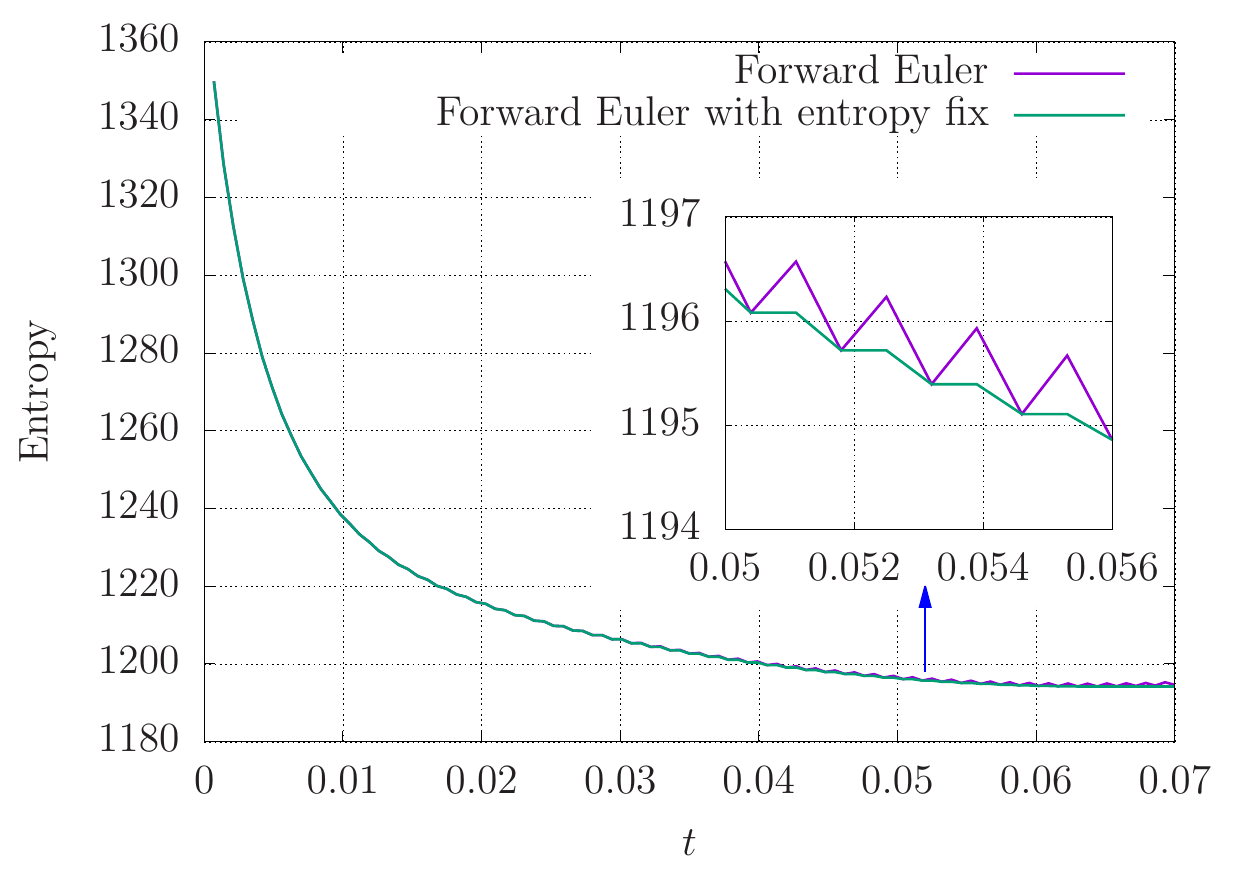}}
\subfigure[$L^2$ relative error vs $t$.]
	{\includegraphics[width=.45\textwidth]{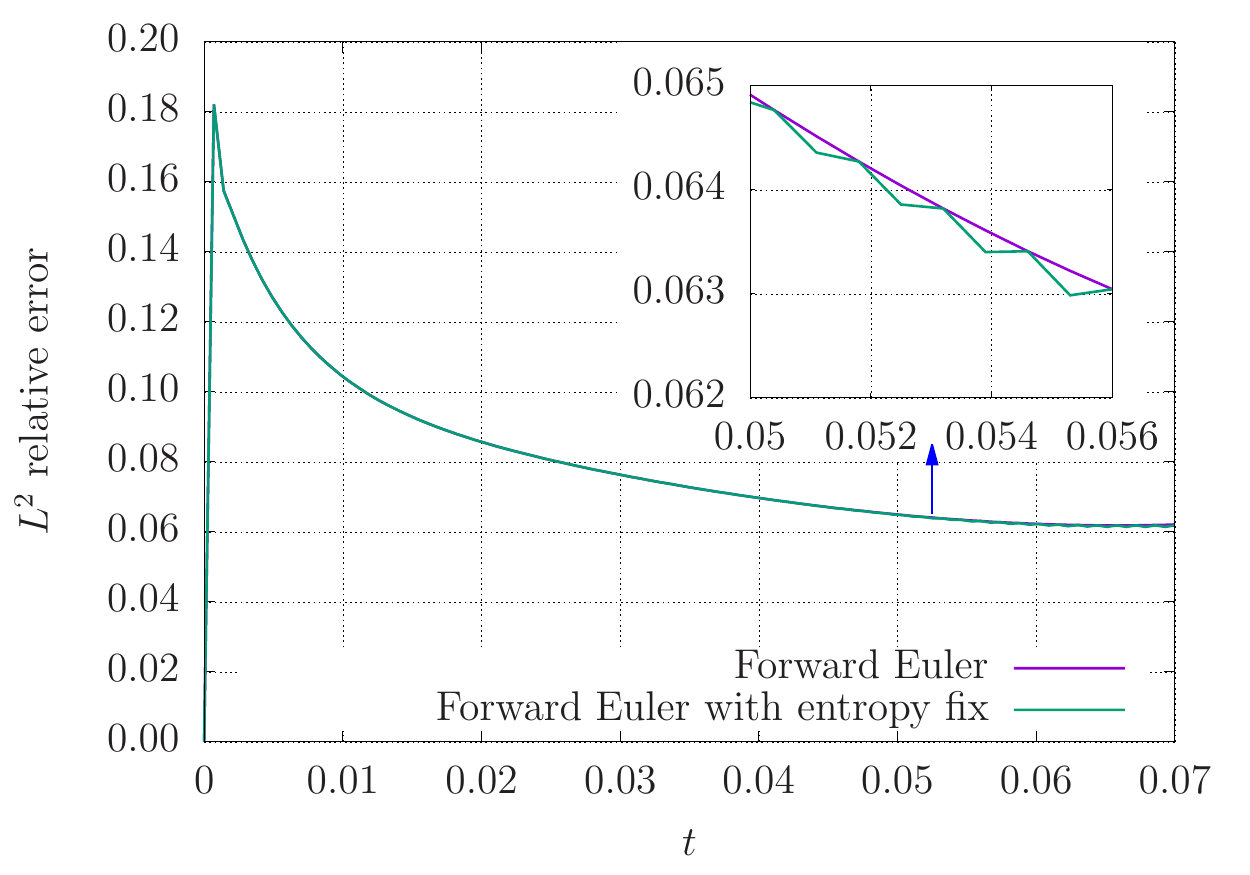}}
	
        \caption{Example of the nonlinear Boltzmann equation. Time evolution of the entropy $H(\boldf)  = \sum_{r \in \mathcal{X} } (f_r \log f_r - f_r) \Delta v $ and the $L^2$ relative error $\| \boldf - \boldf_{\rm exact} \|_2 / \| \boldf_{\rm exact} \|_2 = ( \sum_{r \in \mathcal{X}} ( f_r - f_{{\rm exact},r} )^2 \Delta v )^{1/2} / ( \sum_{r \in \mathcal{X}} ( f_{{\rm exact},r} )^2 \Delta v )^{1/2}$, where $\Delta v = ( 3(3+\sqrt{2}) / 17 )^3$ and $\Delta t=0.0007$. $\boldf_{\rm exact}$ is the numerical solution evaluated with time step $\Delta t = 0.000175$.
        \label{fig:bkw3d}
        }
\end{figure}

% \jh{It might be helpful to include an example of the Fokker-Planck equation but it is up to you. For example, the equation
% \begin{equation}
% u_t=u_{xx}+(V'(x)u)_x.
% \end{equation}
% Let $M(x)=\exp(-V(x))$, then the above equation can be written as
% \begin{equation}
%     u_t=\left(M \left(\frac{u}{M}\right)_x\right)_x.
% \end{equation}
% Further define $g=u/M$, we have
% \begin{equation}
%     g_t=\frac{1}{M}\left(M g_x\right)_x.
% \end{equation}
% If we use the central difference scheme (similarly as in the heat equation), we get
% \begin{equation}
%   \frac{d g_j}{dt}=\frac{1}{M_j}\frac{M_{j+1/2}(g_{j+1}-g_j)-M_{j-1/2}(g_j-g_{j-1})}{h^2}.
% \end{equation}
% For this scheme, we can show
% \begin{equation}
%   \frac{d ( \sum_j M_j g_j\log g_j)}{dt}=\frac{1}{h^2}\sum_j (M_{j+1/2}(g_{j+1}-g_j)-M_{j-1/2}(g_j-g_{j-1}))(\log g_j+1)\leq 0.
% \end{equation}
% }

\section{Conclusions}
\label{sec:conclusions}

This paper focuses on the entropic method for a conservative and positive system of ordinary differential equations. When the numerical solution at the next time step violates the monotonicity of entropy, our entropic method revises it by a linear interpolation to the constant state. The resulting scheme decays the entropy monotonically, while the order of local truncation error has a slight reduction in general. However, in some special cases, the numerical order is proved to be retained after entropic revision. Numerical experiments validate our results. Future work includes the extension of the entropic method to spatially inhomogeneous kinetic equations such as the Boltzmann equation and the radiative transfer equations.

%for a system more than mass conservation. We plan to extend the idea to the system with momentum and energy conservation, which should be retained after entropic revision. Furthermore, the current method could be directly applied to the spatially homogeneous Boltzmann equation, and the modification to the inhomogeneous one is in progress.

\appendix
\section{Proof of \cref{lemma:twoquotients}} \label{appendix:flowbound}
    This proof is composed of three steps:
    \begin{enumerate}
        \item $F(x,y,C) \geq F(x, \frac{1}{2C}, C)$ for $0 \leq x \leq \frac{1}{2}$, $C>1$ and $0 \leq y \leq \frac{1}{2C}$;
        \item $F(x, \frac{1}{2C}, C) \geq \min(F(0, \frac{1}{2C}, C), F(\frac{1}{2}, \frac{1}{2C}, C))$ for $0 \leq x \leq \frac{1}{2}$ and $C>1$;
        \item for any $C_1 \in (0,1]$, there is $C_2 > 1$ depending on $C_1$ such that $F(0, \frac{1}{2C_2}, C_2) \geq \frac{1}{C_1}$ and $F(\frac{1}{2}, \frac{1}{2C_2}, C_2) \geq \frac{1}{C_1}$.
    \end{enumerate}
\subsection{First step}
    It is sufficient to show $\frac{\partial F(x, y, C)}{\partial y} \leq 0$ for $y \geq 0$, from which $F(x,y,C) \geq F(x, \frac{1}{2C}, C)$ for $0 \leq y \leq \frac{1}{2C}$. By the expression of $F(x, y, C)$ in \cref{eq:quotientf}, it could be calculated that
	\begin{equation} \label{eq:fbeta1}
		\frac{\partial F(x, y, C)}{\partial y} = \frac{F_1(x, y, C)}{\left(h(x) - h(x+ y)\right)^2},
	\end{equation}
	where
	%\begin{displaymath}
	%	\begin{split}
	\begin{align*}
			F_1(x, y, C) & = x \log \left(y+x\right) \left(\log \left(x\right)-\log \left(y C+x\right)\right) \\
			& +  C x \log \left(y C+x\right) \left(\log \left(y+x\right)-\log \left(x\right)\right) \\
			& +y C \left(\log \left(y+x\right)-\log \left(y C+x\right)\right).
	\end{align*}
	%	\end{split}
	%\end{displaymath}
	Then we take the derivative of $F_1(x, y, C)$ with respect to $y$,
	\begin{equation} \label{eq:f1beta1}
		\frac{\partial F_1(x, y, C)}{\partial y} = \frac{F_2(x, y, C)}{ \left(y+x \right) \left(y C+x \right)},
	\end{equation}
	where
	%\begin{displaymath}
	%	\begin{split}
	\begin{align*}
			F_2(x, y, C) & = x^2 \left(C^2 \left(\log \left(y+x\right)-\log \left(x\right)\right)-\log \left(y C+x\right)+\log \left(x\right)\right) \\
			& + y C x \left(-2 \log \left(y C+x\right)-C \left(-2 \log \left(y+x\right)+\log \left(x\right)+1\right) \right)  \\
			& + y C x \left(\log \left(x\right)+1 \right) + y^2 C^2 \left(\log \left(y+x\right)-\log \left(y C+x\right)\right).
	\end{align*}
	%	\end{split}
	%\end{displaymath}
	We continue to take the derivative of $F_2(x, y, C)$ w.r.t. $y$,
	%\begin{displaymath}
	%	\begin{split}
	\begin{align*}
			\frac{\partial F_2(x, y, C)}{\partial y} =  C (C-1) x - C \left(-2 C h(x+y)+2 h\left( x + C y \right)+\left(C-1\right) h(x) \right).
	\end{align*}
	%	\end{split}
	%\end{displaymath}
	When $C>1$, the convexity of $h(\cdot)$ implies
	\begin{displaymath}
		h(x+y) \leq \left( 1 - \frac{1}{C} \right) h(x) + \frac{1}{C} h(x + C y).
	\end{displaymath}
	Therefore, 
	\begin{displaymath}
		-2 C h(x+y)+2 h\left( x + C y \right)+\left(C-1\right) h(x) \geq (1-C)h(x).
	\end{displaymath}
	As a result,
	\begin{displaymath}
		\frac{\partial F_2(x, y, C)}{\partial y} \leq C(C-1) (h(x) + x) \leq 0,
	\end{displaymath}
	where the last inequality utilizes $h(x) + x =x \log (x) \leq 0$ when $x \leq \frac{1}{2}$.
	
	$\frac{\partial F_2(x, y, C)}{\partial y} \leq 0$ implies $ F_2(x, y, C)$ is decreasing with respect to $y$ for fixed $x$ and $C$. At the same time, it is easy to verify that $F_2(x, 0, C)=0$. Therefore, $F_2(x, y, C) \leq F_2(x, 0, C)=0$ for $y \geq 0$. 
	
	From \cref{eq:f1beta1} and $F_2(x, y, C) \leq 0$, it is easy to get $\frac{\partial F_1(x, y, C)}{\partial y} \leq 0$, which means $F_1(x, y, C)$ is decreasing with respect to $y$ for fixed $x$ and $C$. Combining with $F_1(x, y, C) \mid _{y=0} = 0$, we could find $F_1(x, y, C) \leq 0$ for $y \geq 0$.
	
	Finally, plugging $F_1(x, y, C) \leq 0$ into \cref{eq:fbeta1}, we could conclude that $\frac{\partial F(x, y, C)}{\partial y} \leq 0$ for $y \geq 0$.

\subsection{Second step}
	For simplicity, We would like to introduce $G(x,C)$ to denote $F(x, \frac{1}{2C}, C)$ as
	\begin{equation} \label{eq:ftog}
	    G(x,C) = F(x, \frac{1}{2C}, C) = \frac{h(x + 1/(2C)) - h(x + 1/2)}{h(x) - h(x + 1/(2C))},
	\end{equation}
	where the second equality is achieved by plugging $y = 1/(2C)$ into \cref{eq:quotientf}. We will show that for fixed $C >1$, $G(x,C)$ is increasing and then decreasing for $0 < x \leq \frac{1}{2}$, from which it is easy to see $G(x,C) \geq \min(G(0,C), G(\frac{1}{2}, C))$. The idea is similar to the first step, which utilizes the sign of derivative.

    By the expression of $G(x,C)$ in \cref{eq:ftog}, a direct calculation shows
	\begin{equation} \label{eq:ffi}
		\frac{\partial G(x, C)}{\partial x} = \frac{G_1(x, C)}{2 C \left(h\left(x\right)-h\left(x + \frac{1}{2 C} \right)\right)^2},
	\end{equation}
	where
	%\begin{displaymath}
	%	\begin{split}
	\begin{align*}
			G_1(x, C) & = C \left(\log \left(x+\frac{1}{2}\right)-1\right) \left(\log \left(x\right)-\log \left(x + \frac{1}{2 C}\right)\right) \\
			& -\left(\log \left(x\right)-\log \left(x+\frac{1}{2}\right)\right) \left(\log \left(x + \frac{1}{2 C}\right)-1\right).
	\end{align*}
	%	\end{split}
	%\end{displaymath}
	Again, we taken the derivative of $G_1(x, C)$ w.r.t. $x$,
	\begin{equation} \label{eq: f1fi}
		\frac{\partial G_1(x, C)}{\partial x} = \frac{G_2(x, C)}{x \left(2 x+1\right) \left(2 C x+1\right)},
	\end{equation}
	where
	%\begin{displaymath}
	%	\begin{split}
	\begin{align*}
			G_2(x, C) & = 4 C^2 x^2 \left(\log \left(x\right)-\log \left(x + \frac{1}{2 C}\right)\right) -\log \left(x + \frac{1}{2 C}\right) \\
			& + 4 C x \left(\log \left(x+\frac{1}{2}\right)-\log \left(x + \frac{1}{2 C}\right)\right)  +1 - C \\
			&+ C \left( - 4 x^2 \left(\log \left(x\right)-\log \left(x+\frac{1}{2}\right)\right)+\log \left(x+\frac{1}{2}\right) \right).
	\end{align*}
	%	\end{split}
	%\end{displaymath}
	We continue to take the derivative of $G_2(x, C)$ w.r.t. $x$,
	\begin{displaymath}
	%	\begin{split}
	%\begin{align*}
			\frac{\partial G_2(x, C)}{\partial x} = 8 C \left(- C h\left(x + \frac{1}{2 C}\right)+ \left(C-1\right) h(x)+ h\left(x+\frac{1}{2}\right)\right).
	%\end{align*}
	%	\end{split}
	\end{displaymath}
	The convexity of $h(\cdot)$ and $C>1$ implies
	\begin{displaymath}
		h\left(x + \frac{1}{2 C}\right) \leq \left( 1 - \frac{1}{C} \right) h(x) + \frac{1}{C}  h\left(x+\frac{1}{2}\right),
	\end{displaymath}
	which means
	\begin{displaymath}
		- C h\left(x + \frac{1}{2 C}\right)+ \left(C-1\right) h(x)+ h\left(x+\frac{1}{2}\right) \geq 0.
	\end{displaymath}
	Therefore, $\frac{\partial G_2(x, C)}{\partial x} \geq 0$ for $0< x \leq \frac{1}{2}$, meaning $G_2(x, C)$ is increasing w.r.t. $x$ for fixed $C$. On the other hand, 
	%\begin{displaymath}
	%	\begin{split}
	\begin{align*}
			\lim_{x \to 0} G_2(x, C) & = -\log \left(\frac{1}{C}\right)-C (1+\log (2))+1+\log (2) \\
			& \leq C -1 -C (1+\log (2))+1+\log (2) = - \log(2) (C -1) < 0,
	\end{align*}
	%	\end{split}
	%\end{displaymath}
	and
	\begin{displaymath}
		G_2(\frac{1}{2}, C) = -C - \left(C+1\right)^2 \log \left(\frac{1}{2 C}+\frac{1}{2}\right)-C \left(C-1\right) \log (2)+1 \geq 0.
	\end{displaymath}
	Therefore, for fixed $C$, there exists $0 < G_2^0 \leq \frac{1}{2}$, such that $G_2(x, C) \leq 0$ for $x \leq G_2^0$ and $G_2(x, C) \geq 0$ for $x \geq G_2^0$. The reason for $G_2(\frac{1}{2}, C) \geq 0$ can be revealed from taking derivatives, i.e., 
	%\begin{displaymath}
	%	\begin{split}
	\begin{align*}
			& \frac{\mathrm{d} G_2(\frac{1}{2}, C)}{\mathrm{d} C} = \frac{1}{C} + 3 \log(2) - 2 \log \left( \left( \frac{C+1}{C} \right)^{(C + 1)} \right), \\
			& \frac{\mathrm{d}^2 G_2(\frac{1}{2}, C)}{\mathrm{d} (C)^2} = \frac{2}{C}-\frac{1}{C^2}-2 \log \left(\frac{1}{C}+1\right), \\
			& \frac{\mathrm{d}^3 G_2(\frac{1}{2}, C)}{\mathrm{d} (C)^3} = \frac{2}{C^4+C^3} > 0.
	\end{align*}
	%	\end{split}
	%\end{displaymath}
	$\frac{\mathrm{d}^3 G_2(\frac{1}{2}, C)}{\mathrm{d} (C)^3} > 0$ implies $\frac{\mathrm{d}^2 G_2(\frac{1}{2}, C)}{\mathrm{d} (C)^2}$ is increasing, which gives
	\begin{displaymath}
		\frac{\mathrm{d}^2 G_2(\frac{1}{2}, C)}{\mathrm{d} (C)^2} \leq \lim_{C \to \infty} \frac{\mathrm{d}^2 G_2(\frac{1}{2}, C)}{\mathrm{d} (C)^2} = 0.
	\end{displaymath} 
	Therefore, $\frac{\mathrm{d} G_2(\frac{1}{2}, C)}{\mathrm{d} C}$ is decreasing, 
	\begin{displaymath}
		\frac{\mathrm{d} G_2(\frac{1}{2}, C)}{\mathrm{d} C} \geq \lim_{C \to \infty} \frac{\mathrm{d} G_2(\frac{1}{2}, C)}{\mathrm{d} C} = 3 \log (2) - 2 > 0.
	\end{displaymath}
	As a result, $G_2(\frac{1}{2}, C)$ is increasing for $C > 1$ and $G_2(\frac{1}{2}, C) \geq G_2(\frac{1}{2}, 1) = 0$.
	
	Since $G_2(x, C) \leq 0$ for $x \leq G_2^0$ and $G_2(x, C) \geq 0$ for $x \geq G_2^0$, we could find $G_1(x, C)$ is decreasing on $(0, G_2^0]$ and increasing on $[G_2^0, \frac{1}{2}]$ from \cref{eq: f1fi}. On the other hand, due to $C>1$ and $\log(C) \leq C -1$,
	\begin{displaymath}
		\lim_{x \to 0} G_1(x, C) = \lim_{x \to 0}  \left((1-C)\log(2) + \log(C)+1-C \right)\log(x) = \infty.
	\end{displaymath}
	Together with
	\begin{displaymath}
		G_1(\frac{1}{2}, C) = \left(C+\log (2)\right) \log \left(\frac{1}{C}+1\right)-\log (2) (1+\log (2)) \leq 0,
	\end{displaymath}
	we could get for fixed $C$, there exists $0 < G_1^0 \leq \frac{1}{2}$, such that $G_1(x, C) \geq 0$ for $x \leq G_1^0$ and $G_1(x, C) \leq 0$ for $x \geq G_1^0$. Similar to $G_2(\frac{1}{2}, C)$, the reason for $G_1(\frac{1}{2}, C) \leq 0$ can be revealed from taking derivatives.
	%\begin{displaymath}
	%	\begin{split}
	\begin{align*}
			& \frac{\mathrm{d} G_1(\frac{1}{2}, C)}{\mathrm{d} C} = \log \left(\frac{1}{C}+1\right)-\frac{C+\log (2)}{C(1+C)}, \\
			& \frac{\mathrm{d}^2 G_1(\frac{1}{2}, C)}{\mathrm{d} (C)^2} = \frac{C (\log (4)-1)+\log (2)}{C^2 \left(C+1\right){}^2} > 0,
	\end{align*}
	%	\end{split}
	%\end{displaymath}
	which means $\frac{\mathrm{d} G_1(\frac{1}{2}, C)}{\mathrm{d} C}$ is increasing w.r.t. $C$. Therefore, 
	\begin{displaymath}
		\frac{\mathrm{d} G_1(\frac{1}{2}, C)}{\mathrm{d} C} \leq \lim_{C \to \infty} \frac{\mathrm{d} G_1(\frac{1}{2}, C)}{\mathrm{d} C} = 0,
	\end{displaymath}
	which implies $G_1(\frac{1}{2}, C)$ is decreasing for $C>1$. Hence, $G_1(\frac{1}{2}, C) \leq G_1(\frac{1}{2}, 1) = 0$.
	
	Using \cref{eq:ffi}, together with $G_1(x, C) \geq 0$ for $x \leq G_1^0$ and $G_1(x, C) \leq 0$ for $x \geq G_1^0$, we could get $G(x,C)$ is increasing on $(0, G_1^0]$ and then decreasing on $[G_1^0, \frac{1}{2}]$ with respect to $x$.
%As a result, $F(f_i^{n+1},C_2) \geq \min(F(0,C_2), F(\frac{1}{2}, C_2))$.

\subsection{Third step}
    With the notation in \cref{eq:ftog}, we would like to evaluate $G(0,C)$ and $ G(1/2, C))$ one by one.
	
	On the one hand, for $G(0, C)$, since $\log (2 C) \leq 2 \sqrt{C} - 1$ for $C \geq 1$ (which can be proved by the monotonicity of $\log (2 C) - 2\sqrt{C} +1 $), it holds that
	\begin{displaymath}
		G(0, C) = \frac{h(\frac{1}{2C}) - h(\frac{1}{2})}{h(0) - h(\frac{1}{2C})} = \frac{C (1+\log (2))}{\log \left(2 C\right)+1}-1 \geq \frac{1+\log (2)}{2}\sqrt{C} -1.
	\end{displaymath}
	Therefore, for any $C_1 \in (0,1]$, we could take $C_2 = \left( \frac{2(1+C_1)}{C_1(1+\log(2))} \right)^2$, which gives $G(0, C_2) \geq \frac{1}{C_1}$. Furthermore, it is easy to find $C_2 = \left( \frac{2}{(1+\log(2))} \right)^2 \left( \frac{1+C_1}{C_1} \right)^2 \geq \frac{16}{(1+\log(2))^2}$ since $\frac{1+C_1}{C_1} \geq 2$ for $0<C_1 \leq 1$. 
	
	On the other hand, for $G(\frac{1}{2}, C)$,
	\begin{displaymath}
		G(\frac{1}{2}, C) = \frac{h(\frac{1}{2} + \frac{1}{2C}) - h(1)}{h(\frac{1}{2}) - h(\frac{1}{2} + \frac{1}{2C})} = \frac{C+\left(C+1\right) \left( \log \left(\frac{1}{C}+1\right) -\log(2) \right)-1}{-\left(C+1\right) \log \left(\frac{1}{C}+1\right)+1+\log (2)}.
	\end{displaymath}
	Since $\left(C+1\right) \log \left(\frac{1}{C} + 1\right) \geq 1$, it holds that when $C \geq \frac{16}{(1+\log(2))^2}$, the numerator
	\begin{displaymath}
		C+\left(C+1\right) \left( \log \left(\frac{1}{C}+1\right) -\log(2) \right)-1 \geq (1 - \log(2)) C -\log(2) > 0.
	\end{displaymath}
	Then, we could utilize $\left(C+1\right) \log \left(\frac{1}{C} + 1\right) \geq 1$ in the denominator of $G(\frac{1}{2}, C)$ and get
	\begin{displaymath}
		G(\frac{1}{2}, C) \geq \frac{(1 - \log(2)) C -\log(2)}{\log (2)}.
	\end{displaymath}
	Therefore, we could take $C_2 = \max(\frac{16}{(1+\log(2))^2}, \frac{(C_1 + 1)\log (2)}{C_1(1 - \log(2))})$ to get $G(\frac{1}{2}, C_2) \geq \frac{1}{C_1}$.
	
	Combining the results of $G(0, C_2)$ and $G(\frac{1}{2}, C_2)$, we could conclude that for any $C_1 \in (0,1]$, there exists $C_2 = \max \Big( \left( \frac{2(1+C_1)}{C_1(1+\log(2))} \right)^2, \frac{(C_1 + 1)\log (2)}{C_1(1 - \log(2))} \Big)$ such that $G(0,C_2) \geq \frac{1}{C_1}$ and $G(\frac{1}{2}, C_2) \geq \frac{1}{C_1}$. In fact, for $C_1 \in (0,1]$, $ \left( \frac{2(1+C_1)}{C_1(1+\log(2))} \right)^2 \geq \frac{(C_1 + 1)\log (2)}{C_1(1 - \log(2))}$. The derivative of their difference is
	%\begin{displaymath}
	%	\begin{split}
	\begin{align*}
			& \frac{\mathrm{d}}{\mathrm{d}C_1} \left( \left( \frac{2(1+C_1)}{C_1(1+\log(2))} \right)^2 - \frac{(C_1 + 1)\log (2)}{C_1(1 - \log(2))} \right) \\
			= & \frac{-C_1 \left(-8+\log ^3(2)+2 \log ^2(2)+\log (512)\right)+8-8 \log (2)}{C_1^3 (\log (2)-1) (1+\log (2))^2}.
	\end{align*}
	%	\end{split}
	%\end{displaymath}
	Since $\left(-8+\log ^3(2)+2 \log ^2(2)+\log (512)\right) < 0$, the above numerator is greater than $8 - 8 \log(2)$ for $0< C_1 \leq 1$, which is positive. Combining with the negative denominator, the above derivative is negative, therefore, 
	\begin{displaymath}
		\left( \frac{2(1+C_1)}{C_1(1+\log(2))} \right)^2 - \frac{(C_1 + 1)\log (2)}{C_1(1 - \log(2))} \geq \left( \frac{4}{1+\log(2)} \right)^2 - \frac{2 \log (2)}{1 - \log(2)} > 0.
	\end{displaymath}
    As a result, $\max \Big( \left( \frac{2(1+C_1)}{C_1(1+\log(2))} \right)^2, \frac{(C_1 + 1)\log (2)}{C_1(1 - \log(2))} \Big) = \left( \frac{2(1+C_1)}{C_1(1+\log(2))} \right)^2$, and the third step is proved with $C_2 = \left( \frac{2(1+C_1)}{C_1(1+\log(2))} \right)^2$.

\section{Coefficients in Eq. \cref{eq:Boltzmann}} \label{appendix:coeff}
The values of $A_{pq}^{rs}$ are given by
\begin{equation}
    \label{eq:coeffApq}
    A_{pq}^{rs} =  \frac{1}{M^{9}} \sum_{l, h, k \in K}  \hat{B}_{M}^{\sigma}(h-k, l-k) E_{-l}(p-s) E_{-h}(q-s) E_{k}(r-s),
\end{equation}
where $K$ is defined as $K=\{ k \mid k =(k_1, k_2, k_3), -m \leq k_1, k_2, k_3 \leq m\}$ with $M = 2m+1$, and $E_k(v) = \exp(\frac{\mathbf{i}\pi}{T} k\cdot v)$ is the Fourier basis on the period $[-T, T]^3$. The kernel function $\hat{B}_{M}^{\sigma}(\cdot,\cdot)$ are defined by
\begin{displaymath}
    \hat{B}_{M}^{\sigma} (i, j):=\hat{B} (i \bmod M, j \bmod M) \sigma_M(i\bmod M)\sigma_M(j\bmod M),
\end{displaymath}
where $\bmod$ is the symmetric modulo function such that each component of $i \bmod M$ ranges from $-m$ to $m$, and $\sigma_M(i)=\tilde{\sigma}_M(i_1)\tilde{\sigma}_M(i_2)\tilde{\sigma}_M(i_3)$ where $\tilde{\sigma}_M(\beta)$ is the one-dimensional modified Jackson filter \cite{Filter2006} given by 
\begin{displaymath}
   \tilde{\sigma}_M(\beta) = \frac{(m+1 - |\beta|)\cos \left( \frac{\pi |\beta|}{m+1} \right) + \sin \left( \frac{\pi |\beta|}{m+1} \right) \cot \left( \frac{\pi}{m+1} \right)}{m+1}.
\end{displaymath}
In the example in \cref{example:boltzmann}, we adopt the kernel modes for the case of the Maxwell molecules presented in \cite{pareschi2000numerical} with
\begin{displaymath}
   \hat{B}(k,l) := \int_{0}^{1} r^{2} \operatorname{Sinc}(\xi r) \operatorname{Sinc}(\eta r) \,\mathrm{d} r  = \frac{(\xi + \eta) \sin (\xi - \eta) - (\xi - \eta) \sin (\xi + \eta) }{2 \xi \eta (\xi^2 - \eta^2)},
\end{displaymath} 
where $\xi=|k+l| \lambda \pi, \eta=|k-l| \lambda \pi$, and $\lambda = 2/(3+\sqrt{2}$). In the numerical simulation, we take $M=17$ and $T = 3/\lambda$. 
%
%\section*{Acknowledgments}
%We would like to acknowledge the assistance of volunteers in putting
%together this example manuscript and supplement. 

\bibliographystyle{siamplain}
\bibliography{references,hu_bibtex}
\end{document}

%% file: boltzmann_shared.tex
\usepackage{lipsum}
\usepackage{amsfonts}
\usepackage{graphicx}
\usepackage{epstopdf}
\usepackage{algorithmic}
\usepackage{booktabs,subfigure} % add by yk
\usepackage{enumitem}
\ifpdf
  \DeclareGraphicsExtensions{.eps,.pdf,.png,.jpg}
\else
  \DeclareGraphicsExtensions{.eps}
\fi

\newsiamremark{remark}{Remark}
\newsiamremark{hypothesis}{Hypothesis}
\crefname{hypothesis}{Hypothesis}{Hypotheses}
\newsiamthm{claim}{Claim}

\headers{An Entropic Method for Discrete Systems with Gibbs Entropy}{Z. Cai, J. Hu, Y. Kuang, and B. Lin}

\title{An Entropic Method for Discrete Systems with Gibbs Entropy\thanks{Submitted to the editors DATE.
\funding{ZC's work was funded by the Academic Research Fund of the Ministry of Education of Singapore under grant Nos.~R-146-000-305-114 and R-146-000-326-112. JH's research is partially supported by NSF CAREER grant DMS-1654152.}}}

\author{Zhenning Cai\thanks{Department of Mathematics, National University of Singapore, Singapore 119076 
  (\email{matcz@nus.edu.sg}).}
\and Jingwei Hu\thanks{Department of Mathematics, Purdue University, West Lafayette, IN 47907, USA
  (\email{jingweihu@purdue.edu}).}
\and Yang Kuang\thanks{Department of Mathematics, National University of Singapore, Singapore 119076 
  (\email{matkuan@nus.edu.sg}).}
\and Bo Lin\thanks{Department of Mathematics, National University of Singapore, Singapore 119076 
	(\email{matbl@nus.edu.sg}).}}

\usepackage{amsopn}

%% file: boltzmann_article.bbl
\begin{thebibliography}{10}

\bibitem{BCH20}
{\sc R.~Bailo, J.~A. Carrillo, and J.~Hu}, {\em Fully discrete
  positivity-preserving and energy-dissipating schemes for
  aggregation-diffusion equations with a gradient flow structure}, Comm. Math.
  Sci., 18 (2020), pp.~1259--1303.

\bibitem{Buet2006numerical}
{\sc C.~Buet and S.~Cordier}, {\em Numerical analysis of conservative and
  entropy schemes for the {F}okker--{P}lanck--{L}andau equation}, SIAM J.
  Numer. Anal., 36 (2006), pp.~953--973.

\bibitem{Cai2018entropic}
{\sc Z.~Cai, Y.~Fan, and L.~Ying}, {\em An entropic fourier method for the
  {B}oltzmann equation}, SIAM Journal on Scientific Computing, 40 (2018),
  pp.~A2858--A2882, \url{https://doi.org/10.1137/17M1127041}.

\bibitem{Chow2019entropy}
{\sc S.~Chow, L.~Dieci, and W.~Li}, {\em Entropy dissipation
  semi-discretization schemes for {F}okker--{P}lanck equations}, J. Dyn. Diff.
  Equat., 31 (2019), pp.~765--792.

\bibitem{Corless1996lambertw}
{\sc R.~M. Corless, G.~H. Gonnet, D.~E.~G. Hare, D.~J. Jeffrey, and D.~E.
  Knuth}, {\em On the {Lambert$W$} function}, Advances in Computational
  Mathematics, 5 (1996), pp.~329--359.

\bibitem{degond1994entropy}
{\sc P.~Degond and B.~Lucquin-Desreux}, {\em An entropy scheme for the
  fokker-planck collision operator of plasma kinetic theory}, Numer. Math., 68
  (1994), pp.~239--262.

\bibitem{Goldstein1989investigations}
{\sc D.~Goldstein, B.~Strutevant, and J.~E. Broadwell}, {\em Investigations of
  the Motion of Discrete-Velocity Gases}, AIAA, 1989, pp.~100--117.

\bibitem{Hoorfar2008inequalities}
{\sc A.~Hoorfar and M.~Hassani}, {\em Inequalities on the {L}ambert {$W$}
  function and hyperpower function}, J. Inequal. Pure and Appl. Math, 9 (2008),
  pp.~1--5.

\bibitem{pareschi2000numerical}
{\sc L.~Pareschi and G.~Russo}, {\em {Numerical solution of the Boltzmann
  equation I: Spectrally accurate approximation of the collision operator}},
  SIAM J. Numer. Anal., 37 (2000), pp.~1217--1245.

\bibitem{Pareschi2018structure}
{\sc L.~Pareschi and M.~Zanella}, {\em Structure preserving schemes for
  nonlinear {F}okker--{P}lanck equations and applications}, J. Sci. Compute.,
  74 (2018), pp.~1575--1600.

\bibitem{nonlinearopt2006}
{\sc A.~Ruszczy\'nski}, {\em Nonlinear optimization}, Princeton University
  Press, Princeton, 2006.

\bibitem{Filter2006}
{\sc A.~Wei{\ss}e, G.~Wellein, A.~Alvermann, and H.~Fehske}, {\em The kernel
  polynomial method}, Rev. Mod. Phys., 78 (2006), pp.~275--306,
  \url{https://doi.org/10.1103/RevModPhys.78.275},
  \url{https://link.aps.org/doi/10.1103/RevModPhys.78.275}.

\end{thebibliography}
